\documentclass[12pt,a4paper]{article}

\usepackage{amssymb, latexsym,amsmath,amsthm}
\usepackage{url}

\usepackage[colorlinks=false, linktocpage=true]{hyperref}

\newcommand{\vare}{\varepsilon}

\newcommand{\p}{\partial}

\newcommand{\e}{\epsilon}

\newcommand{\R}{{\mathbb R}}

\newcommand{\spde}{{\textsc{spde}}}

\newcommand{\cL}{\ensuremath{\mathcal L}}

\newtheorem{theorem}{Theorem}
\newtheorem{lemma}[theorem]{Lemma}
\newtheorem{coro}[theorem]{Corollary}

\newtheorem{assume}[theorem]{Assumption}
\theoremstyle{remark}
\newtheorem{remark}{Remark}
\topsep=\parskip

\title{Diffusion approximation for self-similarity of stochastic advection in Burgers' equation}

 \author{Wei Wang\thanks{Department of
Mathematics, Nanjing University, Nanjing, \textsc{China}.
\protect\url{mailto:wangweinju@aliyun.com} }
\and
A.~J. Roberts\thanks{School of Mathematical Sciences, University of Adelaide, South Australia, \textsc{Australia}. \protect\url{mailto:
anthony.roberts@adelaide.edu.au}}}

\begin{document}

\maketitle
 
\begin{abstract}
Self-similarity of Burgers' equation with stochastic advection is studied.  In self-similar variables a stationary solution is constructed which establishes the existence of a stochastically self-similar solution for the stochastic Burgers' equation.   The analysis assumes that the stochastic coefficient of advection is transformed to a white noise in the self-similar variables.  Furthermore, by a diffusion approximation, the long time convergence to the self-similar solution is proved in the sense of distribution. 
\end{abstract}



\section{Introduction}\label{sec:introduction}
The deterministic Burgers' partial differential equation for a field~$w(t,x)$ is
\begin{equation}\label{e:Burgers}
w_t=\nu w_{xx}-ww_x
\end{equation}
and was proposed by Burgers~\cite{Burgers} to help understand the statistical theory of turbulent fluid motion.
Here $w(t,x)$~is analogous to the velocity field and $\nu$~represents the dissipative viscosity.  To better model the randomness inherent in the presumed chaos of turbulence, the following stochastic Burgers' equation has been suggested~\cite{CAMSS88, CTMK92,  Hay00,  HY75, Jeng69,  Sinai91, Sinai92} and  studied recently by many people~\cite{Bertini94, BD07, Truman, E99, Holden94, Prato94, Prato95,   Truman08}:
\begin{equation}\label{e:tx-Burgers}
w_t=\nu w_{xx}-ww_x+h(t,x,w, w_x)
\end{equation}
where $h(t,x,w, w_{x})$ represents stochastic effects defined on a complete probability space~$(\Omega, \mathcal{F}, \mathbb{P})$.
On bounded domains the existence and uniqueness of global solution was studied by Da Prato et al.~\cite{Prato94} when the noise term~$h(t,x,w,w_{x})$ is white in time, and by Holden et al.~\cite{Holden94} using a white noise calculus. On an unbounded domain the existence of a solution was studied through a Cole--Hopf transformation by Bertini et al.~\cite{Bertini94} with $h(t,x,w,w_{x})$ an additive space-time white noise.

We consider a family of solutions with special spatial-temporal form, namely the family of \emph{self-similar solutions}, of the stochastic Burgers' equation~(\ref{e:tx-Burgers}) on the unbounded real line with the particular stochastic advection $h(t,x,w,w_{x})=(w\zeta)_x$ for some special space-time noise~$\zeta(t,x)$ to be defined later.  Here for the stochastic system, the self-similarity is in the sense of distribution which is defined later.  
The existence of self-similar solutions and the asymptotic   emergence of self-similar solutions comprises the self-similarity of the stochastic Burgers' equation.  Importantly for applications, the form of the stochastic advection,~$(w\zeta)_x$, is appropriate for globally conserved quantities~$w$.   Such stochastic advection is potentially of great interest in applications as it potentially illuminates some of the stochastic nature of chaotic turbulence in fluid flows.  A thorough understanding of turbulence remains an outstanding challenge and researchers are increasingly invoking stochastic terms to model its effects in important environmental applications~\cite[e.g.]{Imkeller02, Frederiksen2008, Zidikheri2008}.  We need to know how stochastic advection affects long term dynamics.

Self-similarity is an important property of systems of physical interest, of which Burgers' equation is a special case. Many researchers have studied the existence of self-similar solutions of deterministic systems~\cite[e.g.]{Xin96, Zuazua91, Liu85, Rob98,  Zuazua94}, and described the asymptotic behavior of self-similar solutions~\cite[e.g.]{Becker, Zuazua91, Miller,  Zuazua94}.  But very little is known about self-similarity in stochastic spatio-temporal systems.
We prove the existence and emergence of self-similar solutions, in the sense of distribution, for the stochastic Burgers' equation~(\ref{e:tx-Burgers}).  The stochastic advection $h=(w\zeta)_x$ for the stochastic Burgers' equation~(\ref{e:tx-Burgers}) transforms to a multiplicative white noise in the following  self-similar variables.  As in earlier research~\cite[e.g.]{Becker, Rob98, WR11} we introduce  log-time and stretched space,
\begin{equation*}
\tau=\log t\,,\quad \xi=\frac{x}{\sqrt{t}}\,,
\quad\text{for }t\geq 1\,,
\end{equation*}
and then define the stochastic field
\begin{equation*}
u(\tau, \xi,\omega)=\sqrt{t}w(t,x,\omega),\quad \omega\in\Omega\,.
\end{equation*}
Straightforward algebra derives  that the \spde~\eqref{e:tx-Burgers} transforms in the similarity variables to
\begin{equation}\label{e:tau-xi-Burgers}
du=\left[\nu u_{\xi\xi}+\frac12\xi u_\xi+\frac12 u-uu_\xi\right]d\tau+(u\,dW)_{\xi}\,.
\end{equation}
We call  a solution $w(t, x,\omega)$  to the stochastic Burgers' equation~(\ref{e:tx-Burgers})  a stochastically self-similar solution if the distribution of  $\sqrt{t}w(t,x,\omega)$  just depends on the self-similar variable $\xi=x/\sqrt{t}$\,.  
By this definition,  any  stationary solution~$\bar{u}(\xi,\omega)$ to equation~(\ref{e:tau-xi-Burgers}) is a stochastically self-similar solution of stochastic Burgers' equation~(\ref{e:tx-Burgers}).
In order to construct a self-similar solution of the stochastic Burgers' equation~(\ref{e:tx-Burgers}) we assume that $W(\tau, \xi,\omega)$~is an $L^2(\R)$~valued Wiener process defined on~$(\Omega,\mathcal{F}, \mathbb{P})$ with covariance operator~$Q$ which is detailed later.

To construct a stationary solution of the transformed \spde~(\ref{e:tau-xi-Burgers}), we consider the system in a weighted space $L^2(K)$ which is defined in the next section~\ref{sec:pre}. First, by using energy estimates and the compact embedding results of the weighted space, we show the tightness of solutions with initial value in the space~$L^2(K)$.  Then the classical Bogolyubov--Krylov method~\cite{Arnold03} implies the existence of a stationary solution of the \spde~(\ref{e:tau-xi-Burgers}).  Due to the multiplicative noise,  the method showing the attraction of the stationary solution in the case of additive noise~\cite{WR11} fails here due to the appearance of the unbounded term~$\dot{W}$.  Instead we apply a diffusion approximation to this stationary solution. For this we introduce the following random equation
\begin{equation}\label{e:rBurgers2}
u^{\e}_{\tau}=\mathcal{L}u^{\e}-u^{\e}u^{\e}_{\xi}-\tfrac12(u^{\e}q)_{\xi\xi}+\tfrac12(u^{\e}q')_{\xi}+\tfrac{1}{\sqrt{\e}}(u^{\e}\bar{\eta}^{\e})_{\xi}
\end{equation}
which is a Burgers' type equation with a singular random perturbation.  Here $\bar{\eta}^{\e}$ is a stationary process solving~(\ref{e:eta}), and $q$ and~$q'$ are introduced in Assumption~\ref{ass:a}.  Then attraction of the stationary solutions is derived  by the method used for the case of additive noise, and  by the approximation of~(\ref{e:rBurgers2}) to~(\ref{e:tau-xi-Burgers}), the attraction is passed to the stationary solution of~(\ref{e:tau-xi-Burgers}).  Here the most difficult part is to show the effectiveness of the approximation. We follow a martingale method to show the tightness of the family of stationary solutions of the approximating model. Then passing to the limit derives the attraction of stationary solutions of the stochastic advection Burgers' equation~(\ref{e:tau-xi-Burgers}). 

However,  for simplicity and intuition of the  discussion in our approach we consider the following  Burgers' type equation 
\begin{equation}\label{e:rBurgers}
u^{\e}_{\tau}=\mathcal{L}u^{\e}-u^{\e}u^{\e}_{\xi} +\frac{1}{\sqrt{\e}}(u^{\e}\bar{\eta}^{\e})_{\xi}\,.
\end{equation}
The limit of the above equation is shown to be the following \spde{}
\begin{equation}\label{e:SPDE}
u_{\tau}=\mathcal{L}u-uu_{\xi}+\tfrac12(uq)_{\xi\xi}-\tfrac12(uq')_{\xi}+(u\dot{\bar{W}})_{\xi}
\end{equation}
for some new Wiener process $\bar{W}$ distributed as~$W$.
By the assumptions on~$q$, all results for equation~(\ref{e:rBurgers})  hold for equation~(\ref{e:rBurgers2}) by exactly same discussion.  Then we have approximation of~(\ref{e:rBurgers2}) to~(\ref{e:tau-xi-Burgers}).

To show the approximation of the stationary solutions, in our approach we consider the statistical solution of the equations~(\ref{e:tau-xi-Burgers}) and~(\ref{e:rBurgers2}).

%
%
%

\section{Preliminary}\label{sec:pre}
We consider the stochastic \textsc{pde}~(\ref{e:tau-xi-Burgers}) in the self-similarity variables.  For brevity we
introduce the linear operator 
\begin{equation*}
\mathcal{L} u=\nu u_{\xi\xi}+\tfrac12\xi u_\xi+\tfrac12 u\,.
\end{equation*}
Denoting the weight function by $K(\xi)=\exp\{\xi^2/4\nu\}$,  
introduce the following weighted  functional space for exponent $p>0$
\begin{equation*}
L^p(K)=\left\{u\in L^p(\R): \|u\|^p_{L^p(K)}=\int_\R|u(\xi)|^pK(\xi)\,d\xi<\infty\right\},
\end{equation*}
and for positive integer exponent~$k$
\begin{equation*}
H^k(K)=\left\{u\in L^2(K): \|u\|^2_{H^k(K)}=\sum_{0\leq\alpha\leq k}\|D^\alpha u\|^2_{L^2(K)}<\infty\right\}\,.
\end{equation*}
For $p=2$\,, denote by $\langle\cdot, \cdot \rangle$ the inner product in space $L^{2}(K)$\,.
Then the linear operator~$\mathcal{L}$ is self-adjoint and  generates an analytic semigroup~$S(\tau)$ on the space~$L^2(K)$ with the domain $D(\mathcal{L})=H^2(K)$~\cite{Kavian}. Further, the eigenvalues of the operator~$\mathcal{L}$ are 
\begin{equation*}
\lambda_k=-\frac{k}{2}\,,\quad k=0, 1,2,\ldots\,,
\end{equation*}
with the corresponding eigenfunctions
\begin{equation*}
e_0(\xi)=\frac{1}{\sqrt{4\pi\nu}}\exp\{-\xi^2/4\nu\}\,,\quad e_k(\xi)=c_k\p_\xi^ke_0(\xi),\quad k=1,2,\ldots\,,
\end{equation*}
which forms a standard orthonormal basis of~$L^2(K)$ when we choose~$c_k$ as the constants such that $\|e_k\|_{L^2(K)}=1$\,. 

In the following we denote by $E_{c}=\text{span}\{e_1(\xi)\}$ and
\begin{equation*}
E_{s}=E_{c}^\bot=\left\{u\in L^2(K): \int_\R u(\xi)\,d\xi=0\right\}.
\end{equation*}
We also denote by $\Pi_{s}$ the linear projection from~$L^{2}(K)$ to~$E_{s}$.

The following are some important basic properties on these weighted spaces~\cite{Kavian}.
\begin{lemma}\label{lem:Kavian}
\begin{enumerate}
  \item The embedding $H^1(K)\subset L^2(K)$ is compact.
  \item There exists $C>0$ such that for any $u\in H^1(K)$
      \begin{equation*}
      \int_\R|u(\xi)|^2|\xi|^2K(\xi)\,d\xi\leq C\int_\R|\nabla u(\xi)|^2K(\xi)\,d\xi\,.
      \end{equation*}
  \item \label{i:3} For any $u\in H^1(K)$,
   \begin{equation*}
    \frac12\int_\R|u(\xi)|^2K(\xi)\,d\xi\leq \int_\R|\nabla u(\xi)|^2K(\xi)\,d\xi\,.
   \end{equation*}
  \item For any $u\in E_{s}$\,,
   \begin{equation*}
    \left\langle \mathcal{L}u , u \right\rangle\leq -\frac12 \|u\|^2_{H^1(K)}\,.
   \end{equation*}
  \item If $u\in H^1(K)$, then $K^{1/2}u\in L^\infty(\R)$.
  \item \label{i:6}  For any $q>2$ and $\e>0$ there exists constants $C_{\e,q}>0$ and $R>0$\,, such that for any $u\in H^{1}(K)\cap L^{q}_{\text{loc}}(\R)$
 \begin{equation*}
 \|u\|^{2}_{L^{2}(K)}\leq \e \|u_{\xi}\|^{2}_{L^{2}(K)}+C_{\e,q} \|u\|^{2}_{L^{q}(B(0, R))}.
 \end{equation*} 

\end{enumerate}
\end{lemma}
\begin{remark}\label{rem:norm}
By item~\ref{i:3} in the above lemma, in the space $H^1(K)$ we define the norm~$\|\nabla u\|_{L^2(K)}$  which is equivalent to $\|u\|_{H^1(K)}$.
\end{remark}
Further, by the spectral properties of the linear operator~$\mathcal{L}$, we define $(-\mathcal{L}+1/2)^\gamma$ for any $\gamma\in\R$~\cite{Yosida}. Then define the Sobolev space~$H^{\gamma}(K)$, for any $\gamma\in\R$\,, as $\mathcal{D}((-\mathcal{L}+1/2)^{\gamma/2})$, the domain of $(-\mathcal{L}+1/2)^{\gamma/2}$. By the embedding theorem~\cite{Yosida}, $H^{\gamma_1}(K)$ is compactly embedding into $H^{\gamma_2}(K)$ for  $\gamma_1>\gamma_2$\,.

We make the following assumptions on the stochastic force.

\begin{assume} \label{ass:a}
 \begin{enumerate}
 \item 
The stochastic force $t\sqrt{t}\eta=t\sqrt{t}(w\zeta)_x$ is written, in the self-similar variables,   as~$(uW)_{\xi}$. Here $W$~is an $L^2(K)$-valued Wiener process   with covariance operator~$Q$ such that
\begin{equation*}
Q\varphi(\xi)=\int_{\R}q(\xi,\zeta)\varphi(\zeta)K(\zeta)\,d\zeta
\quad\text{for any }\varphi\in L^{2}(K),
\end{equation*}
with $q(\xi,\zeta)=q(\zeta, \xi)$ positive, and 
\begin{equation*}
\int_{\R}\int_{\R}q(\xi, \zeta)K_{\xi}(\xi)K_{\zeta}(\zeta)\,d\xi \, d\zeta<\infty\,.
\end{equation*}
The covariance~$Q$ shares the same eigenbasis as that of the operator~$\mathcal{L}$. 
 
\item  $W_{\xi}(\tau,\xi)$  is  an $L^2(K)$-valued Wiener process with covariance operator~$Q'$ such that
\begin{equation*}
Q'\varphi(\xi)=\int_{\R}q'(\xi,\zeta)\varphi(\zeta)\,d\zeta
\end{equation*}
with $q'(\xi,\zeta)=q'(\zeta, \xi)$ positive, and 
\begin{equation*}
\int_{\R}\int_{\R}q'(\xi, \zeta)K(\xi)K(\zeta)\,d\xi \, d\zeta<\infty\,.
\end{equation*}
Furthermore, 
\begin{equation}\label{e:traceQ}
\operatorname{Tr}Q<\infty \quad \text{and} \quad \operatorname{Tr}Q'<\infty\,,
\end{equation}
and  $q(\xi):=q(\xi, \xi)\in H^{2}(K)$,
\begin{equation}\label{e:q-bound}
\|q\|_{L^{\infty}(\R\times \R)} \quad \text{and}\quad \|q'\|_{L^{\infty}(\R\times \R)}\quad \text{are small.}
\end{equation}
 \end{enumerate}
 \end{assume}
From the above assumptions,  $q'(\xi, \zeta)=\p_{\xi}\p_{\zeta}q(\xi, \zeta)$ and the Wiener  process~$W$ has the series representation~\cite{Prato}
\begin{equation}\label{e:W}
W(\tau, \xi)=\sum_{k=0}^{\infty}\sqrt{q_{k}}e_{k}(\xi)\beta_{k}(\tau),
\end{equation}
where $\{\beta_{k}\}_{k}$ are  independent  standard Brownian motions.

%

\begin{remark}
An example of such stochastic force is 
\begin{equation*}
\zeta(x, t, \omega)=\sqrt{t}\sigma\left(\frac{x}{\sqrt{t}}\right)\frac{d}{dt}\beta(t,\omega)
\end{equation*}
where $\frac{d}{dt}\beta(t,\omega)$ is some random process such that $\frac{d}{d\tau}\beta(e^{\tau}, \omega)$
is white in log-time $\tau$\,.
The special assumptions on $\zeta(x,t,\omega)$ does not exclude the existence of self-similar solutions for other cases.
\end{remark}

Recall that  a random process~$\{u(\tau)\}_{\tau\geq 0}$ is said to be stationary if  its joint probability distribution does not change when shifted in time $\tau$~\cite{Arnold03}.  
   For the \spde{}~(\ref{e:tau-xi-Burgers}), to construct  a stationary solution  it is convenient to consider the transition semigroup  associated to  equation~(\ref{e:tau-xi-Burgers}).  We define~$\{P_{\tau}\}_{\tau\geq 0}$ on   
the space consisting of bounded continuous functions $\psi: L^{2}(K)\cap L^{\infty}(\R)\rightarrow \R$ by~\cite{Prato}
\begin{equation}
(P_{\tau}\psi)(u^{0})=\mathbb{E}\psi(u(\tau; u^{0})),
\end{equation}
where $u(\tau; u^{0})$ is the solution of equation~(\ref{e:tau-xi-Burgers}) with initial value $u^{0}\in L^{2}(K)\cap L^{\infty}(\R)$.
 Denote by $\mathcal{M}$ the space consisting  all probability measures on the space $L^{2}(K)\cap L^{\infty}(\R)$ and endow $\mathcal{M}$ with the topology of weak convergence.   Define the dual semigroup $\{P^{*}_{\tau}\}_{\tau\geq 0}$ acting  on $\mathcal{M}$ as 
\begin{equation*} 
 \int_{L^{2}(K)\cap L^{\infty}(\R)} \psi(u)(P^{*}_{\tau}\mu)(du)=\int_{L^{2}(K)\cap L^{\infty}(\R)}(P_{\tau}\psi)(u)\mu(du)
\end{equation*}
for any $\mu\in\mathcal{M}$ and bounded continuous function  $\psi: L^{2}(K)\cap L^{\infty}(\R)\rightarrow \R$\,.  If $\mathcal{L}(u^{0})$,  the distribution of initial values~$u^{0}$, equals~$\mu$,  then $P_{\tau}^{*}\mu$~is the distribution of the solution $u(\tau; u_{0})$~\cite[Proposition 11.1]{Prato}.  Sometimes $\mathcal{M}$ is too large, so we need the smaller space 
\begin{equation*}
\mathcal{M}_{2}=\left\{\mu\in\mathcal{M}: \int_{L^{2}(K)\cap L^{\infty}(\R)}\|u\|^{2}_{L^{2}(K)}\mu(du)<\infty \right\}.
\end{equation*}

A probability space $\mu\in\mathcal{M}$ is said to be a stationary measure for the stochastic Burgers' equation~(\ref{e:tau-xi-Burgers}) if 
\begin{equation*}
P_{\tau}^{*}\mu=\mu\,,\quad \text{for all}\quad t>0\,.
\end{equation*}
The following property of stationary measure is useful~\cite[Proposition 11.5]{Prato}.
\begin{lemma}\label{lem:stat-measure}
If $\mu\in\mathcal{M}$ is a stationary  measure for~(\ref{e:tau-xi-Burgers}) and the initial value~$u^{0}$ is $\mathcal{F}_{0}$~measurable with $\mathcal{L}(u^{0})=\mu$\,, then the solution process~$\bar{u}(\tau; u^{0})$  is a stationary solution to the stochastic Burgers' equation~(\ref{e:tau-xi-Burgers}).
\end{lemma}

\section{Existence of self-similar solutions}
\label{sec:self-similar-sol}

By definition,  a stationary solution to the \spde~(\ref{e:tau-xi-Burgers}) is a stochastically  self-similar solution to the stochastic Burgers' equation~(\ref{e:tx-Burgers}).  Next we construct a stationary solution to the \spde~(\ref{e:tau-xi-Burgers}) from any initial value $u_{0}\in L^{2}(K)\cap L^{\infty}(\R)$.

For any $\tau>0$\,, in the mild sense, the transformed stochastic Burgers' \spde~(\ref{e:tau-xi-Burgers}) is written as
\begin{equation}\label{e:mild-tx-Burgers}
u(\tau)=S(\tau)u_{0}+\int_{0}^{\tau}S(\tau-s)u(s)u_{\xi}(s)\,ds+\int_{0}^{\tau}S(\tau-s)(u(s)dW(s))_{\xi}\,.
\end{equation}
Then by the standard method for the existence of mild solutions to \spde{}s~\cite{Prato} we obtain the following theorem whose proof is given in Appendix~\ref{apd:1}.
\begin{theorem}\label{thm:wellpose}
Assume Assumption~\ref{ass:a} holds.  For any $T>0$ and initial value $u_{0}\in   L^{2}(K)\cap L^{\infty}(\R)$, there is a unique mild solution $u(\tau, \xi)$ to \spde{}~(\ref{e:tx-Burgers}) in $L^{2}(\Omega, C(0, T; L^{2}(K))\cap L^{2}(0, T; H^{1}(K)))$. Moreover this mild solution is also the unique weak solution.
\end{theorem}

We construct a stationary solution by the Bogolyubov--Krylov method. For this we  need some estimates in the spaces~$L^{\infty}(\R)$ and~$L^{2}(K)$.

\subsection{$L^{\infty}(\R)$ estimates}\label{sec:u-L-infty}
We follow the approach for a scalar convection-diffusion equation~\cite{Zuazua94} which was also applied to characterise solutions to a stochastic Burgers' equation with additive noise~\cite{WR11}.

We introduce 
\begin{equation*}
\text{sgn}(u)^+=\begin{cases}
1,& u>0\,, \\ 0,&  u\leq 0\,;
\end{cases}
\quad\text{and}\quad
\text{sgn}(u)^-=\begin{cases}
1,& u<0\,, \\  0,&  u\geq 0\,.
\end{cases}
\end{equation*}
Then for $u\in L^2(\R)$ with $ u_{\xi}(t)\in L^2(\R)$, the integral
\begin{equation*}
\int_\R u_{\xi\xi}\phi(u)\,d\xi=-\int_{\R}u^{2}_{\xi}\phi'(u)\,d\xi\leq 0~(\geq 0)
\end{equation*}
for any nondecreasing~(nonincreasing)~$\phi\in C^1(\R)$. By a density discussion the integral
$\int_\R u_{\xi\xi} \text{sgn}(u)^+\,d\xi\leq 0~(\geq 0)$.
Moreover, the integrals $\int_\R uu_\xi\text{sgn}(u)^\pm\,d\xi=0$
and $\int_\R (\xi u_\xi+u)\text{sgn}(u)^\pm\,d\xi=0$\,. Denote by $u^{\pm}=\text{sgn}(u)^{\pm} u$\,. 
Let $m=\|u_{0}\|_{L^{\infty}(\R)}$.  Then multiplying  $\text{sgn}(u-m)^+$ and $\text{sgn}(u^{\e}-m)^{+}$ on both sides of~(\ref{e:tau-xi-Burgers})  and integrating on $\R\times [0, \tau]$ with $\tau>0$\,, the integral 
\begin{equation*}
\int_\R(u(\tau,\xi)-m)^+\,d\xi\leq 0\,.
\end{equation*}
Therefore, $u(\tau,\xi)\leq m$ for any $\tau>0$\,. Similarly 
$u(\tau,\xi)\geq -m$ for $\tau>0$\,.
Then $\|u(\tau)\|_{L^\infty(\R)}\leq m$  for all $\tau>0$\,.

\subsection{Estimates in the space $H^{1}(K)$
}\label{sec:u-L-2}

We first give a  uniform estimate in the space~$L^{2}(K)$.

Let $u(\tau, \xi)=u_{c}(\tau, \xi)+u_{s}(\tau, \xi)$ with $u_{c}\in E_{c}$ and $u_{s}\in E_{s}$\,. Then  
\begin{eqnarray*}
du_{c}&=&0\,,\\
du_{s}&=&\left[\mathcal{L} u_{s}-\Pi_{s}(uu_{\xi})\right]\,d\tau+\Pi_{s}(udW)_{\xi}\,.
\end{eqnarray*}
So 
\begin{equation*}
u_{c}(\tau, \xi)=u_{c}(0,\xi)=\langle u_{0}, e_{0}\rangle e_{0}(\xi)=\int_{\R}u_{0}(\xi)\,d\xi e_{0}(\xi)
\end{equation*}
 which is totally determined by the mass of tha initial value, namely  $M:=\int_{\R}u_{0}(\xi)\,d\xi$\,.

Notice that $\|\cdot\|_{L^{2}(K)}$ is continuous on space $L^{2}(K)$\,.
Now applying It\^o's formula~\cite{Chow} to $\|u_{s}(\tau)\|^{2}_{L^{2}(K)}$, we  obtain 
\begin{eqnarray*}
\tfrac12\frac{d}{d\tau}\|u_{s}(\tau)\|^{2}_{L^{2}(K)}&\leq &-\tfrac12\|u_s\|^2_{H^1(K)}-\langle uu_{\xi}, u\rangle
\\&&{}
+\tfrac12 \left[\|u_{\xi}\|^{2}_{\mathcal{L}_{2}^{Q}}  +\|u\|^{2}_{\mathcal{L}_{2}^{Q'}}\right] +\langle (u\dot{W})_{\xi}, u\rangle\,.
\end{eqnarray*}
Here $\|\cdot\|_{\mathcal{L}_{2}^{Q}}$ and $\|\cdot\|_{\mathcal{L}_{2}^{Q'}}$ are the norms defined on the Hilbert--Schmidt spaces $\mathcal{L}_{2}(Q^{1/2}L^{2}(K), L^{2}(K))$ and $\mathcal{L}_{2}(Q^{'1/2}L^{2}(K), L^{2}(K))$ respectively~\cite{Prato}.  
 
Noticing that, by Assumption~\ref{ass:a},   $\|q\|_{L^{\infty}(\R\times\R)}$ and $\|q'\|_{L^{\infty}(\R\times \R)}$ are small enough~(\ref{e:q-bound}), there exists some positive constant~$c$ such that 
\begin{eqnarray*}
\tfrac12\frac{d}{d\tau}\|u_{s}(\tau)\|^{2}_{L^{2}(K)}&\leq &-c\|u_s\|^2_{H^1(K)}+c\|u_{c}\|^{2}_{H^{1}(K)}-\langle uu_{\xi}, u\rangle
+\langle (u\dot{W})_{\xi}, u\rangle\,.
\end{eqnarray*}

Integrating by parts yields 
\begin{equation*}
\langle uu_\xi, u\rangle=-\tfrac13\int_{\R}u^{3}K_{\xi}\,d\xi\,.
\end{equation*}
By property~\ref{i:6} in  Lemma~\ref{lem:Kavian}, for any $\vare,\vare'>0$ and $q>2$\,, there exist positive constants $C_{\vare}$, $C_{\vare', q}$ and $R$ such that  
 \begin{eqnarray*}
&&\left|\int_{\R}(u)^{3}K_{\xi}\,d\xi\right|=\tfrac12\left|\int_{\R}u\xi K^{1/2} (u)^{2}K^{1/2}\,d\xi\right|\\
&\leq & \tfrac12\left[\int_{\R}(u)^{2}\xi^{2}K\,d\xi\right]^{1/2}\left[\int_{\R}(u)^{4}K\,d\xi\right]^{1/2}\\
&\leq& C\|u_{\xi}\|_{L^{2}(K)}\|u\|_{L^{2}(K)}\|u\|_{L^{\infty}(\R)}\\
&\leq&3\vare C\|u_{\xi}\|^{2}_{L^{2}(K)}+3C_{\vare}\left[\vare'\|u_{\xi}\|^{2}_{L^{2}(K)}+C_{\vare',q}\|u\|^{2}_{L^{q}(B(0, R))}\right]\|u\|^{2}_{L^{\infty}(\R)}\\&\leq &
3\left[\vare C+\vare' C_{\vare}\|u\|^{2}_{L^{\infty}(\R)}\right]\|u_{\xi}\|^{2}_{L^{2}(K)}+3C_{\vare}C_{\vare', q, R}\|u\|^{4}_{L^{\infty}(\R)}
\end{eqnarray*}
with some positive constant $C_{\vare', q, R}$.
Then
\begin{equation}\label{e:uu-xi}
|\langle uu_{\xi}, u\rangle| \leq \left[ \vare C+ \vare'C_{\vare}\|u\|^{2}_{L^{\infty}(\R)}\right]\|u_{\xi}\|^{2}_{L^{2}(K)}+C_{\vare}C_{\vare', q, R}\|u\|^{4}_{L^{\infty}(\R)}\,.
\end{equation}  
 Then for any~$\vare$ and $\vare'>0$\,, there are positive constants that we still denote by~$C_\vare$ and $C_{\vare', q, R}$ for some positive $q$ and $R$  such that
\begin{eqnarray}
&&\tfrac12\frac{d}{d\tau}\|u_{s}(\tau)\|^{2}_{L^{2}(K)}\nonumber\\&\leq& -c\|u_{s}\|^{2}_{H^{1}(K)}+c\|u_{c}\|_{H^{1}(K)}+\left[ \vare C+ \vare'C_{\vare}\|u\|^{2}_{L^{\infty}(\R)}\right]\|u_{\xi}\|^{2}_{L^{2}(K)}\nonumber\\&&{}+C_{\vare}C_{\vare', q, R}\|u\|^{4}_{L^{\infty}(\R)}+\langle (u\dot{W})_{\xi}, u\rangle\nonumber\\
&\leq&\left[-c+\vare C+\vare' C_{\vare}\|u\|_{L^{\infty}(\R)}\right]\|u_{s}\|^{2}_{H^{1}(K)}\nonumber\\
&&{}+\left[c+\vare C+\vare' C_{\vare}\|u\|_{L^{\infty}(\R)}\right]\|u_{c}\|^{2}_{H^{1}(K)}\nonumber\\&&{}+C_{\vare} C_{\vare',q, R}\|u\|^{4}_{L^{\infty}(\R)}+ \langle (u\dot{W})_{\xi}, u\rangle\,.\label{e:u-s}
\end{eqnarray}

Now choose $\vare$ and $\vare'>0$ small enough, and since $u_{c}=Me_{0}(\xi)$ and  $\|u_{c}\|_{H^{1}(K)}\leq CM$ for some $C>0$\,, then by the Gronwall lemma
\begin{equation*}
\mathbb{E} \|u(\tau)\|^{2}_{L^{2}(K)\cap L^{\infty}(\R)}\leq R_{1}\,, \quad\text{for all } \tau\geq 0\,,
\end{equation*}
with some positive random variable~$R_{1}$.

To show the existence of an invariant measure we further need the estimates of $u$ in space $H^{1}(K)$\,.     
From~(\ref{e:u-s}) we have for some constant $C_{1}$\,, $C_{2}>0$ such that 
\begin{equation*}
\mathbb{E}\int_{0}^{\tau}\|u(s)\|_{H^{1}(K)}\,ds\leq C_{1} \tau+C_{2}\,. 
\end{equation*}
By the Chebyshev inequality 
\begin{eqnarray*}
&&\frac{1}{T}\int_{0}^{T}\mathbb{P}(\|u(\tau)\|_{H^{1}(K)}>K)\,d\tau\\&\leq&\frac{1}{K^{2}T}\int_{0}^{T}\mathbb{E}\|u(\tau)\|^{2}_{H^{1}(K)}\,d\tau\\
&\leq& \frac{1}{K^{2}T}(C_{1}T+C_{2})\rightarrow 0\,,\quad K\rightarrow\infty\,.
\end{eqnarray*}
Then by the Krylov--Bogoliubov method~\cite{Arnold03} $P_{\tau}$ has   at least one stationary  measure denoted by $\mu$\,. Let $\bar{u}(\tau,\xi)$~be the solution with initial value distributing as~$\mu$, then $\bar{u}(\tau, \xi)$~is a stationary solution to the transformed \spde~(\ref{e:tau-xi-Burgers}) with distribution~$\mu$~(Lemma~\ref{lem:stat-measure}),  and by the construction of the stationary solution, we have $\bar{u}\in H^{1}(K)$.   Moreover, 
\begin{equation*}
\int_{\R}\bar{u}\,d\xi=\int_{\R}u_{0}\,d\xi \quad \text{and}\quad \|\bar{u}\|_{L^{\infty}(\R)}\leq m\,.
\end{equation*}


\begin{theorem}\label{thm:self-similar-sol-exist}
Assume Assumption~\ref{ass:a} holds. For any initial $u_{0}\in L^{2}(K)\cap L^{\infty}(\R)$, there is a stationary solution, denoted by~$\bar{u}$, such that $\bar{u}\in H^{1}(K)$. Further, there is a sequence~$\tau_{n}$, with $\tau_{n}\rightarrow\infty$ as $n\rightarrow\infty$\,, such that
\begin{equation*}
u(\tau_{n}) \text{ converges in distribution to } \bar{u} \text{ in } L^{2}(K)
\end{equation*}
as $n\rightarrow\infty$\,. Here $u(\tau)$ is the solution to the \spde~(\ref{e:tau-xi-Burgers}) with initial value~$u_{0}$.
\end{theorem}

Next we show that  the above convergence holds for any sequence~$\tau_{n}$ with $\tau_{n}\rightarrow\infty$ as $n\rightarrow \infty$\,; that is, $u(\tau)$ converges in distribution to~$\bar{u}$ as $\tau\rightarrow\infty$\,.  It is impractical to follow the approach used for the stochastic Burgers' equation with additive noise~\cite{WR11}. As Section~\ref{sec:introduction} states,  we consider in the next section a Burgers' equation with random fast fluctuating advection which is an approximation to the \spde~(\ref{e:tau-xi-Burgers}).


\section{Burgers' equation with random fast fluctuations: self-similar solution and stability}
We consider the following randomly fluctuating advection in a Burgers' type equation~(\ref{e:rBurgers}). The random force $\bar{\eta}^{\e}(\tau,\xi)=\bar{\eta}(\tau/\e, \xi)$ in which $\bar{\eta}(\tau, \xi)$ is the stationary  Ornstein--Uhlenbeck  process solving 
\begin{equation}\label{e:eta}
d\eta=- \eta\,d\tau+dW(\tau,\xi).
\end{equation}
Then 
\begin{equation}\label{e:q}
\mathbb{E}\bar{\eta}^{\e}(\tau,\xi)\bar{\eta}^{\e}(s, \zeta)=\tfrac12 q(\xi, \zeta)\exp\left(-\frac{|\tau-s|}{\e}\right),
\end{equation} 
and for $0<\tau\leq s$
\begin{equation}\label{e:con-exp}
\mathbb{E}[\bar{\eta}^{\e}(\tau,\xi)\mid\mathcal{F}_{s}]=\bar{\eta}^{\e}(\tau, \xi)\exp\left(-\frac{\tau-s}{\e} \right).
\end{equation}
Moreover, $\bar{\eta}^{\e}(\tau)\in  H^{1}(K)$ for any $\tau>0$\,.
By the assumption on~$W_{\xi}(\tau, \xi)$, the process $\bar{\eta}^{\e}_{\xi}(\tau, \xi)=\bar{\eta}_{\xi}(\tau/\e, \xi)$ which solves
\begin{equation}\label{e:eta-xi}
d\eta_{\xi}=- \eta_{\xi}\,dt+ dW_{\xi}(\tau,\xi)
\end{equation}
and  also $\bar{\eta}^{\e}_{\xi}(\tau)\in H^{1}(K)$ for any $\tau>0$ with 
\begin{equation}\label{e:q'}
\mathbb{E}\bar{\eta}^{\e}_{\xi}(\tau,\xi)\bar{\eta}^{\e}_{\xi}(s, \zeta)=\tfrac12 q'(\xi, \zeta)\exp\left(-\frac{|\tau-s|}{\e}\right).
\end{equation}
For initial value we assume $u^{\e}_{0}\in L^{2}(\Omega, L^{2}(K)\cap L^{\infty}(\R))$ and  
\begin{equation}\label{e:ini-mM}
\int_{\R}u_{0}^{\e}\,d\xi=M\,, \quad \|u_{0}^{\e}\|_{L^{\infty}(\R)}=m
\end{equation}
for some deterministic positive  constants $m$ and $M$\,.

We study the dynamics of the random differential equation~(\textsc{rde})~(\ref{e:rBurgers}) for fixed $\e>0$\,.   
First, for any $\tau>0$\,, in the mild sense the \textsc{rde}~(\ref{e:rBurgers}) is written as
\begin{equation}\label{e:rmild}
u^{\e}(\tau)=S(\tau)u^{\e}_{0}+\int_{0}^{\tau}S(\tau-s)u^{\e}(s)u_{\xi}^{\e}(s)\,ds+\frac{1}{\sqrt{\e}}\int_{0}^{\tau}S(\tau-s)(u^{\e}(s)\bar{\eta}^{\e}(s))_{\xi}\,ds\,.
\end{equation}
Then by same discussion for equation~(\ref{e:tau-xi-Burgers}) this theorem follows.
\begin{theorem}
Assume Assumption~\ref{ass:a} holds. For any $T>0$ and initial $u^{\e}_{0}\in L^{2}(\Omega,  L^{2}(K)\cap L^{\infty}(\R))$ satisfying~(\ref{e:ini-mM}),  the \textsc{rde}~(\ref{e:rBurgers})  has a unique mild solution 
\begin{equation*}
u^{\e}\in L^{2}\left(\Omega, L^{2} \left(0, T; H^{1}(K) \right)\cap C \left(0, T; L^{2}(K) \right)\right).
\end{equation*}
\end{theorem}

We next construct a stationary solution for the \textsc{rde}~(\ref{e:rBurgers}), which is attractive for any $\e>0$\,.    We follow the discussion in section~\ref{sec:self-similar-sol}.

First, by the same discussion as in section~\ref{sec:u-L-infty},  
\begin{equation*}
\|u^{\e}(\tau)\|_{L^{\infty}(\R)}\leq m \quad \text{for all } \tau>0\,,
\end{equation*}
 with $m=\|u^{\e}_{0}\|_{L^{\infty}(\R)}$\,. Next we give a uniform estimate in~$\tau$  in the space~$H^{1}(K)$ from the   estimate in space~$L^{2}(K)$ for~$u^{\e}$ with any fixed $\e>0$\,. We follow the same discussion as in section~\ref{sec:u-L-2}.

Let $u^{\e}(\tau, \xi)=u^{\e}_{c}(\tau, \xi)+u^{\e}_{s}(\tau, \xi)$ with $u_{c}\in E_{c}$ and $u_{s}\in E_{s}$\,. Then  
\begin{eqnarray}\label{e:u-e-c}
du^{\e}_c&=&0\,,\\
du^{\e}_s&=&\left[\cL  u^{\e}_s-\Pi_{s}(u^{\e}u^{\e}_\xi)\right]d\tau+\tfrac{1}{\sqrt{\e}}(u^{\e}\bar{\eta}^{\e})_{\xi}\,.\label{e:u-e-s}
\end{eqnarray}
So 
\begin{equation*}
u^{\e}_{c}(\tau, \xi)=u^{\e}_{c}(0,\xi)=\langle u^{\e}_{0}, e_{0}\rangle e_{0}(\xi)=\int_{\R}u^{\e}_{0}(\xi)\,d\xi e_{0}(\xi)
\end{equation*}
which is totally determined by the mass of the initial value $M:=\int_{\R}u^{\e}_{0}(\xi)\,d\xi$\,.

Now by multiplying $u^{\e}$ in the space~$L^{2}(K)$ on both sides of the \textsc{rde}~(\ref{e:rBurgers}), 
\begin{equation*}
\tfrac12\frac{d}{d\tau}\|u^{\e}_s\|^2_{L^2(K)}\leq -\tfrac12\|u^{\e}_s\|^2_{H^1(K)}-\langle u^{\e}u^{\e}_{\xi}, u^{\e}\rangle +\tfrac{1}{\sqrt{\e}}\langle (u^{\e}\bar{\eta}^{\e})_{\xi}, u^{\e}\rangle \,.
\end{equation*}
Integrating by parts,
\begin{equation*}
\langle u^{\e}u^{\e}_\xi, u^{\e}\rangle=-\tfrac13\int_{\R}(u^{\e})^{3}K_{\xi}\,d\xi\,.
\end{equation*}
Then by the same discussion as for~(\ref{e:uu-xi}) we have 
that for any~$\vare,\vare'>0$ and $q>2$\,, there exist positive constants~$C_{\vare}$, $R$ and~$C_{\vare', q, R}$ such that 
\begin{eqnarray*}
|\langle u^{\e}u^{\e}_{\xi}, u^{\e}\rangle|\leq \left[ \vare C+ \vare'C_{\vare}\|u^{\e}\|^{2}_{L^{\infty}(\R)}\right]\|u^{\e}_{\xi}\|^{2}_{L^{2}(K)}+C_{\vare}C_{\vare', q, R}\|u^{\e}\|^{4}_{L^{\infty}(\R)}\,.
\end{eqnarray*} 
Moreover,  for any $\vare>0$ there is some positive constant~$C_{\vare}$ such that 
\begin{eqnarray*}
\left|\langle (u^{\e}\bar{\eta}^{\e})_{\xi}, u^{\e}\rangle\right|
=\left|\langle u^{\e}\bar{\eta}^{\e}, u_{\xi}^{\e}\rangle\right|\leq C_{\vare}\|\bar{\eta}^{\e}\|^{2}_{L^{2}(K)}\|u^{\e}\|^{2}_{L^{\infty}(\R)}+\vare\|u^{\e}_{\xi}\|^{2}_{L^{2}(K)}\,. 
\end{eqnarray*}
Then for any~$\vare$ and $\vare'>0$\,, there are positive constants that we still denote by~$C_\vare$ and~$C_{\vare', q, R}$ for some positive $q$~and~$R$  such that
\begin{eqnarray*}
&&\tfrac12\frac{d}{d\tau}\|u^{\e}_{s}(\tau)\|^{2}_{L^{2}(K)}\\&\leq& -\tfrac12\|u^{\e}_{s}\|^{2}_{H^{1}(K)}+\left[ \vare C+ \vare'C_{\vare}\|u^{\e}\|^{2}_{L^{\infty}(\R)}\right]\|u^{\e}_{\xi}\|^{2}_{L^{2}(K)}\\&&{}+C_{\vare}C_{\vare', q, R}\|u^{\e}\|^{4}_{L^{\infty}(\R)}\\&&{}+\tfrac{1}{\sqrt{\e}}\left[ C_{\vare}\|\bar{\eta}^{\e}\|^{2}_{L^{2}(K)}\|u^{\e}\|^{2}_{L^{\infty}(\R)}+\vare\|u^{\e}_{\xi}\|^{2}_{L^{2}(K)}\right]\\
&\leq&\left[-\tfrac12+\vare C+\vare' C_{\vare}\|u^{\e}\|_{L^{\infty}(\R)}+\tfrac{\vare}{\sqrt{\e}} \right]\|u_{s}^{\e}\|^{2}_{H^{1}(K)}\\
&&{}+\left[\vare C+\vare' C_{\vare}\|u^{\e}\|_{L^{\infty}(\R)}+\tfrac{\vare}{\sqrt{\e}} \right]\|u_{c}^{\e}\|^{2}_{H^{1}(K)}\\&&{}+C_{\vare} C_{\vare',q, R}\|u^{\e}\|^{4}_{L^{\infty}(\R)}+\tfrac{1}{\sqrt{\e}}C_{\vare}\|\bar{\eta}^{\e}\|^{2}_{L^{2}(K)}\|u^{\e}\|^{2}_{L^{\infty}(\R)}\,.
\end{eqnarray*}
Now choose $\vare$ and $\vare'>0$ small enough and since $u_{c}^{\e}=Me_{0}(\xi)$,  $\|u^{\e}\|_{L^{\infty}}\leq m$ and  $\|u_{c}^{\e}\|_{H^{1}(K)}\leq CM$ for some $C>0$\,, then  by the Gronwall lemma and the properties of~$\bar{\eta}^{\e}$,
\begin{equation*}
 \mathbb{E}\|u^{\e}(\tau,\omega)\|_{L^{2}(K)\cap L^{\infty}(\R)}\leq R_{1}^{\e}  \quad\text{for all } \tau\geq 0\,,
\end{equation*}
with some positive constant~$R_{1}^{\e}$.

Similarly to the discussion in subsection~\ref{sec:u-L-2}, 
\begin{equation*}
\mathbb{E}\int_{0}^{\tau}\|u^{\e}(s)\|_{H^{1}(K)}\,ds\leq R^{\e}_{2}\,\tau+R^{\e}_{3}
\end{equation*}
for some constants $R^{\e}_{2}$ and $R^{\e}_{3}$\,. Then we show the existence of a stationary measure.
Notice that $u^{\e}$  is not a  Markov process, we consider the process~$(u^{\e}, \bar{\eta}^{\e})$ which is a Markov process on the space $(L^{2}(K)\cap L^{\infty}(\R))\times H^{1}(K)$. For this we introduce the space~$\mathfrak{M}$ consisting all probability measures on the space~$(L^{2}(K)\cap L^{\infty}(\R))\times H^{1}(K)$ and endow the space~$\mathfrak{M}$ with the topology of weak convergence.  Denote  by~$\{\mathcal{P}^{\e}_{\tau}\}_{\tau\geq 0}$ the continuous Markov semigroup  associating with~$(u^{\e}, \bar{\eta}^{\e})$ on~$\mathfrak{M}$  and the dual semigroup as
 \begin{equation*}
\mathcal{P}^{*\e}_{\tau}\mathfrak{y}(A)=\mathbb{P}\{(u^{\e}(\tau,\cdot), \bar{\eta}^{\e}(\tau))\in A\}\,,
\end{equation*}
for any Borel set $A\subset (L^{2}(K)\cap L^{\infty}(\R))\times H^{1}(K)$ and $\mathfrak{y}\in\mathcal{M}$ with form $\mathfrak{y}=\mu*\bar{\nu}$\,.  Here, $u^{\e}(\tau, \cdot)$ is  the solution to equations~(\ref{e:rBurgers}) with initial value distributing as~$\mu$, and $\bar{\nu}$ is the distribution of~$\bar{\eta}^{\e}$.

Now by the compact embedding $H^{1}(K)\subset L^{2}(K)$ and the classical Bo- golyubov--Krylov method~\cite{Arnold03}, a stationary measure exists, denoted by~$\bar{\mathfrak{y}}^{\e}=\bar{\mu}^{\e}*\bar{\nu}$\,.  

Next we show the emergence of the stationary measure~$\bar{\mathfrak{y}}^{\e}$ by showing it is attractive. We follow the discussion for deterministic systems~\cite{Zuazua94, Kim01} which was also applied to the case of additive noise in a stochastic Burgers' equation~\cite{WR11}. 
The following lemma is a key step.
\begin{lemma}\label{lem:contraction}
For any $u_{1,0},u_{2,0}\in L^{2}(K)\cap L^{\infty}(\R)$ with 
\begin{equation*}
\int_{\R}u_{1,0}(\xi)\,d\xi=\int_{\R}u_{2,0}(\xi)\,d\xi\,.
\end{equation*}
Let $u^{\e}_{1}(\tau,\xi)$ and $u^{\e}_{2}(\tau,\xi)$ be the solutions to \textsc{rde}~(\ref{e:rBurgers}) with initial value~$u_{1,0}$ and~$u_{2,0}$ respectively. Then the  function 
\begin{equation*}
\phi^{\e}(\tau)=\int_{\R}|u_{1}^{\e}(\tau,\xi)-u_{2}^{\e}(\tau,\xi)|\,d\xi
\end{equation*}
is strictly decreasing almost surely.
\end{lemma}
\begin{proof}
Let $U^{\e}(\tau,\xi)=u^{\e}_{1}(\tau,\xi)-u_{2}^{\e}(\tau,\xi)$, then it satisfies the following linear equation
\begin{equation}\label{e:U-eps}
U^{\e}_{\tau}=U^{\e}_{\xi\xi}-\tfrac12[(u^{\e}_{1}+u^{\e}_{2}-\xi-\tfrac{1}{\sqrt{\e}}\bar{\eta}^{\e})U^{\e}]_{\xi}\,.
\end{equation}
Notice that for any solution~$u^{\e}$ of the \spde~(\ref{e:rBurgers}), let $v^{\e}=u^{\e}_{\xi}$\,, then
\begin{equation*}
v_\tau^{\e}=\mathcal{L}v^{\e}+\tfrac12 v^{\e}-(v^{\e})^{2}-u^{\e}v_\xi^{\e}+\tfrac{1}{\sqrt{\e}}(u^{\e}\bar{\eta}^{\e})_{\xi\xi}\,.
\end{equation*}
By the same discussion as in section~\ref{sec:u-L-infty}, and the construction of~$\bar{\eta}^{\e}$,
\begin{equation*}
(u^{\e}_{1}(\tau,\xi)+u^{\e}_{2}(\tau,\xi)-\xi-\tfrac{1}{\sqrt{\e}}\bar{\eta}^{\e}(\tau, \xi))_{\xi}
\end{equation*}
 is bounded by a random constant for any $\tau>0$\,. Then for almost all fixed $\omega\in\Omega$\,, the equation~(\ref{e:U-eps}) is a linear equation with bounded coefficient, then  the result follows by the discussion for the corresponding deterministic system~\cite{Zuazua94,Kim01}.
\end{proof}

Now we study the attracting property of any stationary measure.    
Further, we  introduce the following subspace of~$\mathfrak{M}$; define
\begin{equation*}
\mathfrak{M}_{2}=\left\{\mathfrak{y}\in\mathfrak{M}: \int_{(L^{2}(K)\cap L^{\infty}(\R))\times H^{1}(K)}(\|u\|^{2}_{L^{2}(K)}+\|\eta\|^{2}_{H^{1}(K)})\mathfrak{y}(d(u,\eta))<\infty \right\}.
\end{equation*}
%
We next show that for any $\mathfrak{y}\in\mathfrak{M}_{2}$,  with form $\mathfrak{y}=\mu*\bar{\nu}$\,, there is a stationary measure $\bar{\mathfrak{y}}^{\e}=\bar{\mu}^{\e}*\bar{\nu}\in\mathfrak{M}_{2}$ such that $\mathcal{P}^{*\e}_{\tau}\mathfrak{y}$ converges weakly to $\bar{\mathfrak{y}}^{\e}$ as $\tau\rightarrow \infty$\,.

Associated with the solution to the \textsc{rde}~(\ref{e:rBurgers}) we choose  $\mu\in\mathcal{M}_{2}$ which has the form
\begin{equation}\label{e:split}
\mu=\delta_{M}*\mu_{s}\,,
\end{equation}
where $\delta_{M}$ is some Dirac measure on~$E_{c}$ and $\mu_{s}$~is supported on~$E_{s}$.  Then consider the limit of $\mathcal{P}^{*\e}_{\tau}\mathfrak{y}$ as $\tau\rightarrow\infty$ with $\mathfrak{y}=\mu*\nu$\,. First by the  approach of the Bogolyubov--Krylov method, 
 we have a probability measure~$\bar{\mu}^{\e}$,  and subsequence~$\tau_{n}$ with $\tau_{n}\rightarrow\infty$\,, $n\rightarrow\infty$, such that 
\begin{equation}\label{e:convergence}
\mathcal{P}^{*\e}_{\tau_{n}}\mathfrak{y}\rightarrow \bar{\mathfrak{y}}^{\e}:=\bar{\mu}^{\e}*\bar{\nu}\,, \quad n\rightarrow\infty\,.
\end{equation} 
Suppose $\bar{\mu}'^{\e}$ is another  probability measure  such that for some $\tau_{n}'\rightarrow \infty$\,, $n\rightarrow\infty$\,, 
\begin{equation}
\mathcal{P}^{*\e}_{\tau_{n}'}\mathfrak{y}\rightarrow \bar{\mathfrak{y}}'^{\e}:=\bar{\mu}'^{\e}*\bar{\nu}\,,\quad n\rightarrow\infty\,.
\end{equation}
Denote by $\bar{u}^{\e}(\tau,\xi)$ and $\bar{u}'^{\e}(\tau,\xi)$ the two solutions of \textsc{rde}~(\ref{e:rBurgers}) with initial value~$\bar{u}^{1}(\xi)$ and~$\bar{u}^{2}(\xi)$, distributed as $\bar{\mu}^{\e}$~and~$\bar{\mu}'^{\e}$ respectively. Then
\begin{equation*}
\int_{\R}\bar{u}^{1}(\xi)\,d\xi=\int_{\R}\bar{u}^{2}(\xi)\,d\xi\,.
\end{equation*}
By Lemma~\ref{lem:contraction}, the function 
\begin{equation*}
\int_{\R}|\bar{u}^{\e}(\tau,\xi)-\bar{u}'^{\e}(\tau,\xi)|\,d\xi
\end{equation*}
is almost surely strictly decreasing in~$\tau$ which contradicts the stationarity of $\bar{u}^{\e}$~and~$\bar{u}'^{\e}$. Hence we deduce the following theorem. 
\begin{theorem}\label{thm:global}
Assume Assumption~\ref{ass:a} holds. 
For any initial $u^{\e}_{0}\in L^{2}(\Omega, L^{2}(K)\cap L^{\infty}(\R))$ satisfying~(\ref{e:ini-mM}), the solution~$u^{\e}(\tau,\xi)$ to the \textsc{rde}~(\ref{e:rBurgers}), converges in distribution, as $\tau\rightarrow\infty$\,,  to~$\bar{u}^{\e}$ in the space~$L^{2}(K)$ which is the unique solution to \textsc{rde}~(\ref{e:rBurgers}) with 
\begin{equation*}
\int_{\R}\bar{u}^{\e}(\tau,\xi)\,d\xi=\int_{\R}u^{\e}_{0}(\xi)\,d\xi\,.
\end{equation*}
\end{theorem}
\begin{remark}
By the construction of~$\bar{u}^{\e}$,  
\begin{equation}\label{e:u-bar-eps}
\|\bar{u}^{\e}(\tau)\|_{L^{\infty}(\R)}\leq m\,,\quad \mathbb{E}\|\bar{u}^{\e}(\tau)\|^{2}_{H^{1}(K)}\leq C\quad \text{for all } \tau\geq 0\,,
\end{equation}
for some constant~$C>0$. 
\end{remark}
We want to pass the above convergence  property to the stochastic Burgers' equation~(\ref{e:tau-xi-Burgers}); that is, we want to pass to the limit $\e\rightarrow 0$ in~$u^{\e}$ in the space~$C([0, \infty), L^{2}(K))$. We give some estimates uniform in~$\e$ in the next section.


\section{Some a priori estimates on finite time intervals}\label{sec:estimates}

This section shows the tightness of $\{u^{\e}\}_{0<\e\leq 1}$ in the space~$C(0, T; L^{2}(K))$ for any $T>0$\,. We assume $u_{0}^{\e}\in L^{2}(\Omega, L^{2}(K)\cap L^{\infty}(\R))$ and converges in distribution to $u_{0}\in L^{2}(\Omega, L^{2}(K)\cap L^{\infty}(\R))$\,.

From equation~(\ref{e:u-e-s}),
by the chain rule, 
\begin{eqnarray}\label{e:energy}
 \tfrac12\frac{d}{dt}\|u_{s}^{\e}(\tau)\|^{2}_{L^{2}(K)} 
&\leq&-\tfrac12\|u^{\e}(\tau)\|^{2}_{H^{1}(K)}+\left\langle u^{\e}(\tau)u^{\e}_{\xi}(\tau) , u^{\e}(\tau)\right\rangle \\
&&{}+\left\langle \tfrac{1}{\sqrt{\e}}\bar{\eta}^{\e}(\tau)u^{\e}_{\xi}(\tau)+\tfrac{1}{\sqrt{\e}}\bar{\eta}_{\xi}^{\e}(\tau)u^{\e}(\tau), u^{\e}(\tau)\right\rangle.\nonumber
\end{eqnarray}
Define the two integrals
\begin{eqnarray*}
I^{\e}_{1}(\tau)&=&\frac{1}{\sqrt{\e}}\int_{0}^{\tau}\langle u^{\e}_{\xi}(s)\bar{\eta}^{\e}(s), u^{\e}(s)\rangle\,ds\,, \\ 
I^{\e}_{2}(\tau)&=&\frac{1}{\sqrt{\e}}\int_{0}^{\tau}\langle u^{\e}(s)\bar{\eta}_{\xi}^{\e}(s), u^{\e}(s)\rangle\,ds\,.
\end{eqnarray*}
Now by the factorization method, for some $0<\alpha<1$\,,
\begin{equation*}
\frac{1}{\sqrt{\e}}\int_{0}^{\tau}\langle u^{\e}_{\xi}(s)\bar{\eta}^{\e}(s), u^{\e}(s)\rangle\,ds
=\frac{\sin\pi\alpha}{\alpha}\int_{0}^{\tau}(\tau-s)^{\alpha-1}Y^{\e}(s)\,ds 
\end{equation*}
where 
\begin{eqnarray*}
Y^{\e}(s)&=&\frac{1}{\sqrt{\e}}\int_{0}^{s}(s-r)^{-\alpha}\langle  u^{\e}_{\xi}(r)\bar{\eta}^{\e}(r), u^{\e}(r) \rangle\,dr\\
&=&\frac{1}{2\sqrt{\e}}\int_{0}^{s}(s-r)^{\alpha}\langle (u^{\e}(r))^{2}_{\xi}, \bar{\eta}^{\e}(r) \rangle\,dr\\
&=&-\frac{1}{2\sqrt{\e}}\int_{0}^{s}(s-r)^{-\alpha}\int_{\R}(u^{\e}(r))^{2}\bar{\eta}^{\e}(r)K_{\xi}\,d\xi \, dr\\&&{}-\frac{1}{2\sqrt{\e}}\int_{0}^{s}(s-r)^{-\alpha}\int_{\R} (u^{\e}(r))^{2}\bar{\eta}^{\e}_{\xi}(r)K\,d\xi\,dr\\
&=&Y_{1}^{\e}(s)+Y^{\e}_{2}(s).
\end{eqnarray*}
Then for any $T>0$\,, there is some positive constant $C_{T}$ such that 
\begin{equation*}
\sup_{0\leq \tau\leq T}|I_{1}^{\e}(\tau)|^{2}\leq C_{T}\int_{0}^{T}|Y_{1}^{\e}(s)|^{2}\,ds+C_{T}\int_{0}^{T}|Y_{2}^{\e}(s)|^{2}\,ds\,.
\end{equation*}
 We first consider~$Y_{1}^{\e}$. By  the $L^{\infty}(\R)$~estimates on~$u^{\e}$, and the construction of~$\bar{\eta}^{\e}$, 
 \begin{eqnarray*}
&& \mathbb{E}|Y_{1}^{\e}(s)|^{2}\\&=&
 \frac{1}{\e}\left|\mathbb{E}\int_{0}^{s}\int_{\rho}^{s}
 \left[(s-r)^{-\alpha}(s-\rho)^{-\alpha}\int_{\R}(u^{\e}(r,\xi))^{2}\bar{\eta}^{\e}(r,\xi)K_{\xi}\,d\xi
 \right.\right.\\&&{}\times\left.\left.
 \int_{\R}(u^{\e}(\rho,\zeta))^{2}\bar{\eta}^{\e}(\rho,\zeta)K_{\zeta}\,d\zeta\right] dr\, d\rho\right|\\
 &\leq& \frac{m^{4}}{\e}\left|\int_{0}^{s}\int_{\rho}^{s}(s-r)^{-\alpha}(s-\rho)^{-\alpha} \int_{\R}\int_{\R}\mathbb{E}\bar{\eta}^{\e}(r,\xi)\bar{\eta}^{\e}(\rho,\zeta)K_{\xi}K_{\zeta}\,d\xi \, d\zeta\, dr\, d\rho\right|\\
 &\leq& C_{1,T}\,,
 \end{eqnarray*}
and similarly 
  \begin{eqnarray*}
 &&\mathbb{E}|Y_{2}^{\e}(s)|^{2}\\&=&
 \frac{1}{\e}\left|\mathbb{E}\int_{0}^{s}\int_{\rho}^{s}
 \left[(s-r)^{-\alpha}(s-\rho)^{-\alpha}\int_{\R}(u^{\e}(r,\xi))^{2}\bar{\eta}_{\xi}^{\e}(r,\xi)K\,d\xi
 \right.\right.\\&&{}\times\left.\left.
 \int_{\R}(u^{\e}(\rho,\zeta))^{2}\bar{\eta}_{\zeta}^{\e}(\rho,\zeta)K\,d\zeta\right] dr\, d\rho\right|\\
 &\leq& C_{1,T}\,,
 \end{eqnarray*}
 for some positive constant~$C_{1,T}$. Then  
 \begin{equation*}
 \mathbb{E}\sup_{0\leq \tau\leq T}|I^{\e}_{1}(\tau)|\leq C_{T}C_{1,T}\,.
 \end{equation*}
By the same discussion for~$I_{2}^{\e}$, a similar expectation holds: 
 \begin{equation*}
 \mathbb{E}\sup_{0\leq \tau\leq T}|I^{\e}_{2}(\tau)|\leq C_{T}C_{2,T}\,,
 \end{equation*}
for some positive constant~$C_{2,T}$.
 Then  by the same discussion for~(\ref{e:uu-xi}) and the Gronwall lemma,
 \begin{equation}\label{e:estimates}
\mathbb{E}\sup_{0\leq \tau\leq T}  \|u^{\e}(\tau)\|^{2}_{L^{2}(K)}+\mathbb{E}\int_{0}^{T}\|u^{\e}(s)\|^{2}_{H^{1}(K)}\,ds\leq C_{T} 
 \end{equation}
 for some positive constant~$C_{T}$.  Notice that in the mild sense
 \begin{eqnarray*}
 u^{\e}(\tau)&=&S(\tau)u^{\e}_{0}+\int_{0}^{\tau}S(
 \tau-\sigma)u^{\e}(\sigma)u^{\e}_{\xi}(\sigma)\,d\sigma\\&&{}+\frac{1}{\sqrt{\e}}\int_{0}^{\tau}S(\tau-\sigma)(u^{\e}(\sigma)\bar{\eta}^{\e}(\sigma))_{\xi}\,d\sigma \,.
 \end{eqnarray*}
Then for any $T>\tau>\delta>0$\,,  
\begin{eqnarray}\label{e:u-tau-delta}
&&\|u^{\e}(\tau)-u^{\e}(\delta)\|_{L^{2}(K)}\\&\leq &\|(S(\tau)-S(\delta))u_0\|_{L^{2}(K)}+\left\|\int_{\delta}^{\tau}S(\tau-\sigma)u^{\e}(\sigma)u^{\e}_{\xi}(\sigma)\,d\sigma\right\|_{L^{2}(K)}\nonumber\\
&&{}+\frac{1}{\sqrt{\e}}\left\|\int_{\delta}^{\tau}S(\tau-\sigma)(u^{\e}(\sigma)\bar{\eta}^{\e}(\sigma))_{\xi}\,d\sigma\right\|_{L^{2}(K)}\nonumber\\&&{}+\left\|\int_{0}^{\delta}[S(\tau-\sigma)-S(\delta-\sigma)]u^{\e}(\sigma)u^{\e}_{\xi}(\sigma)\,d\sigma\right\|_{L^{2}(K)}\nonumber\\
&&{}+\frac{1}{\sqrt{\e}}\left\|\int_{0}^{\delta}[S(\tau-\sigma)-S(\delta-\sigma)](u^{\e}(\sigma)\bar{\eta}^{\e}(\sigma))_{\xi}\,d\sigma\right\|_{L^{2}(K)}.\nonumber
\end{eqnarray}
By the $L^{\infty}(\R)$ estimate of~$u^{\e}$, and by estimate~(\ref{e:estimates}), the expectation
\begin{eqnarray*}
&&\mathbb{E}\left\|\int_{\delta}^{\tau}S(\tau-\sigma)u^{\e}(\sigma)u^{\e}_{\xi}(\sigma)d\sigma\right\|_{L^{2}(K)}\\
&\leq&\mathbb{E}\int_{\delta}^{\tau}\|S(\tau-\sigma)u^{\e}(\sigma)u^{\e}_{\xi}(\sigma)\|_{L^{2}(K)}d\sigma\\
&\leq &\mathbb{E}\int_{\delta}^{\tau}\|u^{\e}(\sigma)u^{\e}_{\xi}(\sigma)\|_{L^{2}(K)}d\sigma
 \\&\leq& m\mathbb{E}\int_{\delta}^{\tau}\|u^{\e}_{\xi}(\sigma)\|_{L^{2}(K)}\,d\sigma\leq  mC_{T}\sqrt {\tau-\delta}\,.
\end{eqnarray*}
Expanding by~$\{e_{k}\}_{k}$ and by~(\ref{e:q}), 
\begin{eqnarray*}
&&\frac{1}{\e}\mathbb{E}\left\| \int_{\delta}^{\tau}S(\tau-\sigma)(u^{\e}(\sigma)\bar{\eta}^{\e}(\sigma))_{\xi}\,d\sigma\right\|^{2}_{L^{2}(K)}
\\&\leq&\frac{1}{\e}\mathbb{E}\sum_{k}\int_{\delta}^{\tau}\int_{\delta}^{\tau} 
\left[e^{-\lambda_{k}(\tau-\sigma)}\int_{\R} u^{\e}(\sigma,\xi)\bar{\eta}^{\e}(\sigma,\xi) (e_{k}K)_{\xi}\,d\xi
\right.\\
&&\left.{}\times e^{-\lambda_{k}(\tau-\lambda)}\int_{\R}u^{\e}(\lambda, \zeta)\bar{\eta}^{\e}(\lambda, \zeta)(e_{k}K)_{\zeta}\,d\zeta\right] d\sigma\, d\lambda \\
&\leq &\frac{m^{2}}{\e} \sum_{k}\int_{\delta}^{\tau}\int_{\delta}^{\tau} 
\left[e^{-\lambda_{k}(\tau-\sigma)}e^{-\lambda_{k}(\tau-\lambda)} \int_{\R} \int_{\R}\mathbb{E}\bar{\eta}^{\e}(\sigma,\xi) \bar{\eta}^{\e}(\lambda, \zeta) 
\right.\\&&\left.\vphantom{\int}{}
\times(e_{k}K)_{\xi}(e_{k}K)_{\zeta}\,d\xi\,d\zeta \right]  d\sigma\, d\lambda\\
&\leq&C_{T}(\tau-\delta),
\end{eqnarray*}
for some positive constant~$C_{T}$. 
By the strong continuity of the semigroup~$S(\tau)$  and a similar discussion, from~(\ref{e:u-tau-delta})  the expectation
\begin{equation}\label{e:u-holder}
\mathbb{E}\|u^{\e}(\tau)-u^{\e}(\delta)\|_{L^{2}(K)}\leq C_{T}\sqrt{\tau-\delta}\,.
\end{equation}

Now we need the following lemma~\cite{Mai03}. Suppose $\mathcal{X}_{1}$~and~$\mathcal{X}_{2}$ are two Banach spaces. Let $T>0$\,, $1\leq p\leq \infty$\,, and $\mathcal{B}$~be a compact operator from~$\mathcal{X}_{1}$ to~$\mathcal{X}_{2}$; that is, $\mathcal{B}$~maps bounded sets of~$\mathcal{X}_{1}$ to relatively compact subsets of~$\mathcal{X}_{2}$. 
\begin{lemma}[\cite{Mai03}]\label{lem:compact}
Let $H$ be a bounded subset of $L^{1}(0, T; \mathcal{X}_{1})$ such that $G=\mathcal{B}H$ is a subset of $L^{p}(0, T; \mathcal{X}_{2})$ bounded in $L^{r}(0, T; \mathcal{X}_{2})$ with $r>1$\,. If 
\begin{equation*}
\lim_{\sigma\rightarrow 0}\|u(\cdot+\sigma)-u(\cdot)\|_{L^{p}(0, T; \mathcal{X}_{2})}=0 \quad \text{uniformly for } u\in G\,,
\end{equation*}
then $G$~is relatively compact in $L^{p}(0, T; \mathcal{X}_{2})$ (and in $C(0, T; \mathcal{X}_{2})$ if $p=+\infty$).
\end{lemma}

Let  $\mathcal{X}_{1}=H^{1}(K)$, $\mathcal{X}_{2}=L^{2}(K)$ and $\mathcal{B}$~be the embedding from $\mathcal{X}_{1}$ to $\mathcal{X}_{2}$\,, by Lemma~\ref{lem:compact}, from estimates~(\ref{e:estimates}) and~(\ref{e:u-holder}) we obtain the following main theorem of this section.
 \begin{theorem}\label{thm:tight-T}
 Assume Assumption~\ref{ass:a} holds. For any $T>0$\,,  and $u^{\e}_{0}\in L^{2}(\Omega, L^{2}(K)\cap L^{\infty}(\R))$ satisfying~(\ref{e:ini-mM}),  the distribution of~$\{u^{\e}\}_{0<\e\leq 1}$ is tight in the space~$C(0, T; L^{2}(K))$.
 \end{theorem}


\section{Diffusion approximation}\label{sec:diffusionApp}

This section determines the limit of~$u^{\e}$, the solutions of the \textsc{rde}~(\ref{e:rBurgers}), as $\e\rightarrow 0$\,.   We first show the tightness of~$u^{\e}$ in the space~$C([0, \infty), L^{2}(K))$ by  Theorem~\ref{thm:tight-T}. Then we determine the limit  of~$u^{\e}$ in the space~$C([0,\infty), L^{2}(K))$ by a martingale approach. 

\subsection{Tightness in space $C([0, \infty), L^{2}(K))$}
We need the following result on the tightness of a family processes~\cite[Theorem~3.9.1]{EK86}.
\begin{lemma}\label{lem:tight}
Let $\mathcal{X}$ be a Polish space and let $\{X^{\e}\}_{0< \e\leq 1}$ be a family of processes with sample paths in the space~$C([0,\infty), \mathcal{X})$. Suppose that for any $\delta>0$ and $T>0$ there exists a compact set $\Gamma_{\delta, T}\subset \mathcal{X}$ such that for all $0<\e\leq 1$
\begin{equation}\label{e:tight-condition}
\mathbb{P}\{ X^{\e}(t)\in \Gamma_{\delta, T} \text{ for } 0\leq t\leq T \}\geq 1-\delta \,.
\end{equation}
Then $\{X^{\e}\}$ is tight in the space $C([0,\infty), \mathcal{X} )$ if and only if $\{F(X^{\e})\}_{0<\e\leq 1}$ is tight in the space~$C([0, \infty), \R)$ for any $F\in C_{b}(\mathcal{X})$, where \(C_{b}(\mathcal{X})\)~is the space consisting of all continuous and bounded functions on~$\mathcal{X}$.
\end{lemma}
\begin{remark}\label{rem:tight}
We do not need to verify the tightness of $\{F(X^{\e})\}_{0< \e\leq 1}$ for all $F\in C_{b}(\mathcal{X})$.  One just needs to verify the tightness for all~$F$ in a dense subset of~$C_{b}(\mathcal{X})$ in the topology of uniform convergence on compact sets~\cite[Theorem~3.9.1]{EK86}.
\end{remark}

By Theorem~\ref{thm:tight-T}, the pre-condition~(\ref{e:tight-condition}) in Lemma~\ref{lem:tight} holds.  Next we show the tightness of $\{F(u^{\e})\}_{0<\e\leq 1}$ in the space~$C([0, \infty))$ for any $F\in C_{b}(L^{2}(K))$. We follow a martingale approach. 
Continue to let $\mathcal{X}$~be a Polish space and $\{X^{\e}\}_{0< \e\leq 1}$~be a family of processes valued in the space~$C([0, \infty), \mathcal{X})$ adapted to the filtration~$\mathcal{F}_{\tau}^{\e}$.
%
%
%
Let $\mathfrak{L}^{\e}$~be the Banach space of real valued $\mathcal{F}_{\tau}^{\e}$-progressive processes with norm $\|Y\|_{\mathfrak{L}^\e}=\sup_{\tau\geq 0}\mathbb{E}|Y(\tau)|$.  Let 
 \begin{equation}
 \mathcal{M}^{\e}=\left\{ (Y, Z)\in \mathfrak{L}^{\e}\times \mathfrak{L}^{\e}: Y(\tau)-\int_{0}^{\tau}Z(s)\,ds \text{ is } \mathcal{F}^{\e}_{\tau}\text{-martingale}  \right\}.
 \end{equation}
Then the following lemma applies~\cite[Theorem~3.9.4]{EK86}.
\begin{lemma}\label{lem:f-tight}
For any bounded continuous function  $F$ on $\mathcal{X}$ with bounded support, and for any $\delta>0$ and $T>0$, there is $(Y^{\e}, Z^{\e})\in \mathcal{M}^{\e}$ such that 
\begin{equation}
\limsup_{\e\rightarrow 0}\mathbb{E}\left[\sup_{\tau\in [0, T]}
\left|Y^{\e}(\tau)-F(X^{\e}(\tau))\right|\right]<\delta
\end{equation}
and 
\begin{equation}
\sup_{\e}\mathbb{E}\left[ \|Z^{\e}\|_{L^{p}(0, T)} \right]<\infty  \quad \text{for some } p\in (1, \infty]. 
\end{equation}
 Then $\{F(X^{\e})\}_{0<\e\leq 1}$ is tight in $C([0, \infty), \R)$.
\end{lemma}
By Remark~\ref{rem:tight} we just need to show the tightness of  the following family of real valued processes~\cite{Kushner}:
\begin{equation*}
 \{f(\langle u^{\e}, \varphi\rangle)\}_{0<\e\leq 1}
\end{equation*}
for any $\varphi\in\mathcal{D}(\R)$ and twice differentiable compactly supported functions~$f$.
From the \textsc{rde}~(\ref{e:rBurgers})  
\begin{eqnarray}\label{e:f}
&&f(\langle u^{\e}(\tau), \varphi \rangle )-
f(\langle u^{\e}_{0}, \varphi \rangle )\\&=&\int_{0}^{\tau}f'(\langle u^{\e}(s), \varphi \rangle)\langle u^{\e}(s), \mathcal{L}\varphi \rangle \,ds-\int_{0}^{\tau}f'(\langle u^{\e}(s), \varphi\rangle)\langle u^{\e}(s)u^{\e}_{\xi}(s), \varphi \rangle \,ds\nonumber\\
&&{}+\frac{1}{\sqrt{\e}}\int_{0}^{\tau}f'(\langle u^{\e}(s), \varphi\rangle)\langle (u^{\e}(s)\bar{\eta}^{\e}(s))_{\xi}, \varphi\rangle \,ds\,.\nonumber
\end{eqnarray}
One can see that the singular term in the above equation is difficult to treat.  To treat this term  we follow a perturbation approach developed by Kushner~\cite{Kushner}.  Let $\mathcal{F}_{\tau}^{\e}$ be the $\sigma$-algebra generated by $\{\bar{\eta}^{\e}(s): 0\leq s\leq \tau\}$. Then introduce the process 
\begin{equation} 
F^{\e}_{1}(\tau)=\frac{1}{\sqrt{\e}}\mathbb{E}\left[ \int_{\tau}^{\infty}f'(\langle u^{\e}(\tau), \varphi\rangle)\langle (u^{\e}(\tau)\bar{\eta}^{\e}(s))_{\xi}, \varphi\rangle ds \,\Big|\,\mathcal{F}_{\tau}^{\e}\right].
\end{equation}
For the process $F_{1}^{\e}(\tau)$ the following lemma holds. 
\begin{lemma}\label{lem:F-1}
Assume Assumption~\ref{ass:a} holds.  Then 
\begin{equation}\label{e:F-1}
F_{1}^{\e}(\tau)=\sqrt{\e}f'(\langle u^{\e}(\tau),\varphi \rangle)\langle (u^{\e}(\tau)\bar{\eta}^{\e}(\tau))_{\xi}, \varphi\rangle\,.
\end{equation}
Furthermore, 
\begin{equation*}
\mathbb{E}|F_{1}^{\e}(\tau)|\leq C\sqrt{\e}
\end{equation*}
for some positive constant~$C$, and  for any $T>0$ 
\begin{equation*}
\mathbb{E}\sup_{0\leq \tau\leq T}|F_{1}^{\e}(\tau)|\rightarrow 0 \quad \text{as } 
\e\rightarrow 0\,.  
\end{equation*} \end{lemma}
\begin{proof}
The equality~(\ref{e:F-1}) is implied by~(\ref{e:con-exp}) and the property of conditional expectation. Then by the $L^{\infty}$~bound on~$u^{\e}$ and the estimates on~$\bar{\eta}^{\e}$,  
\begin{eqnarray*}
\mathbb{E}|F_{1}^{\e}(\tau)|&\leq& \sqrt{\e}\|f'\|_{L^{\infty}(\R)}\|u^{\e}(\tau)\|_{L^{\infty}(\R)}\mathbb{E}\|\bar{\eta}^\e(\tau)\|_{L^{2}(K)}\big\|\varphi_{\xi}+\tfrac12\xi\varphi\big\|_{L^{2}(K)}\\
&\leq&C\sqrt{\e}
\end{eqnarray*}   
for some positive constant~$C$.  Further, by 
the maximal estimate on stochastic integral~\cite[Lemma~7.2]{Prato}, for any $T>0$
\begin{equation*}
\mathbb{E}\sup_{0\leq \tau\leq T}\|\bar{\eta}^{\e}(\tau)\|^{2}_{L^{2}(K)}\leq C_{T}
\end{equation*}
for some positive constant~$C_{T}$. Then by~(\ref{e:F-1}) 
\begin{equation*}
\mathbb{E}\sup_{0\leq \tau\leq T}|F_{1}^{\e}(\tau)|\rightarrow 0 \quad \text{as } 
\e\rightarrow 0\,.
\end{equation*}
The proof is complete.
\end{proof}
To apply Lemma~\ref{lem:f-tight}, we first construct $(Y^{\e}, Z^{\e})\in \mathfrak{L}^{\e}$.   For this we introduce the operator~$A^{\e}$ defined by
\begin{equation}\label{A-eps}
A^{\e}f(\tau)=\mathbb{P}-\lim_{\delta\rightarrow 0}\tfrac{1}{\delta}\left[\mathbb{E}f(\tau+\delta)-f(\tau)\mid\mathcal{F}_{\tau}^{\e}\right]
\end{equation}
for any $\mathcal{F}^{\e}_{\tau}$ measurable function $f$ with $\sup_{\tau}\mathbb{E}|f(\tau)|<\infty$\,.
Then Ethier and Kurtz's proposition~\cite[Proposition~2.7.6]{EK86} yields that 
\begin{equation*}
f(\tau)-\int_{0}^{\tau}A^{\e}f(s)\,ds
\end{equation*}
is a martingale with respect to~$\mathcal{F}^{\e}_{\tau}$.
Now define  $(Y^{\e}, Z^{\e})$ as 
\begin{equation*}
Y^{\e}(\tau)=f(\langle u^{\e}(\tau), \varphi\rangle)+F_{1}^{\e}(\tau),\quad Z^{\e}(\tau)=A^{\e}Y^{\e}(\tau).
\end{equation*}
Then we establish the following lemma. 
\begin{lemma}\label{lem:Z}
\begin{eqnarray*}
Z^{\e}(\tau)&=&f'(\langle u^{\e}(\tau), \varphi\rangle) \langle u^{\e}(\tau), \mathcal{L}\varphi\rangle-f'(\langle u^{\e}(\tau), \varphi\rangle)\langle \tfrac12(u^{\e}(\tau))^{2}_{\xi}, \varphi\rangle
\\&&{}
+f''(\langle u^{\e}(\tau), \varphi\rangle)\langle u^{\e}(\tau)\bar{\eta}^{\e}(\tau), \varphi_{\xi}+\tfrac12\xi\varphi\rangle^{2}
\\&&{}
-f'(\langle u^{\e}(\tau), \varphi\rangle)\langle (u^{\e}(\tau)\bar{\eta}^{\e}(\tau))_{\xi}, (\varphi_{\xi}+\tfrac12\xi\varphi)\bar{\eta}^{\e}(\tau)\rangle
\\&&{}
+\sqrt{\e}f''(\langle u^{\e}(\tau), \varphi\rangle)[\langle u^{\e}(\tau), \mathcal{L}\varphi\rangle-\langle \tfrac12(u^{\e}(\tau))^{2}_{\xi}, \varphi\rangle]\langle (u^{\e}\eta^{\e})_{\xi}, \varphi \rangle
\\&&{}
-\sqrt{\e}f'(\langle u^{\e}(\tau), \varphi\rangle)\left[\langle u^{\e}(\tau), \mathcal{L}((\varphi_{\xi}+\tfrac12\xi\varphi)\bar{\eta}^{\e}(\tau))\rangle 
\right.\\&&\quad{}
-\left.\langle\tfrac12(u^{\e}(\tau))^{2}_{\xi}, (\varphi_{\xi}+\tfrac12\xi\varphi)\bar{\eta}^{\e}(\tau)\rangle\right].
\end{eqnarray*}
\end{lemma}
\begin{proof}
By~(\ref{e:f}),   
\begin{eqnarray*}
A^{\e}f(\langle u^{\e}(\tau), \varphi)
&=&f'(\langle u^{\e}(\tau), \varphi \rangle)\langle u^{\e}(\tau), \mathcal{L}\varphi \rangle \\&&{} -f'(\langle u^{\e}(\tau), \varphi\rangle)\langle u^{\e}(\tau)u^{\e}_{\xi}(\tau), \varphi \rangle \\&&{} +\tfrac{1}{\sqrt{\e}}f'(\langle u^{\e}(\tau), \varphi\rangle)\langle (u^{\e}(\tau)\bar{\eta}^{\e}(\tau))_{\xi}, \varphi\rangle\,.
\end{eqnarray*}
Now consider $A^{\e}F_{1}^{\e}$. By~(\ref{e:F-1}) and the construction of $\bar{\eta}^{\e}$\,,  
\begin{eqnarray*}
&&\mathbb{E}[F_{1}^{\e}(\tau+\delta) \mid \mathcal{F}_{\tau}^{\e}]
\\&=&-\sqrt{\e}\mathbb{E}\big\{[f'(\langle u^{\e}(\tau+\delta), \varphi\rangle)-f'(\langle u^{\e}(\tau), \varphi\rangle)]
\\&&{}
\times\langle u^{\e}(\tau+\delta)\bar{\eta}^{\e}(\tau+\delta), \varphi_{\xi}+\tfrac12\xi\varphi\rangle \mid \mathcal{F}_{\tau}^{\e}\big\}
\\&&{}
-\sqrt{\e}f'(\langle u^{\e}(\tau), 
\varphi\rangle)\mathbb{E}[\langle (u^{\e}(\tau+\delta)-u^{\e}(\tau))\bar{\eta}^{\e}(\tau+\delta), \varphi_{\xi}+\tfrac12\xi\varphi\rangle \mid \mathcal{F}_{\tau}^{\e}]
\\&&{}
-\sqrt{\e}f'(\langle u^{\e}(\tau), \varphi\rangle)\langle u^{\e}(\tau)\bar{\eta}^{\e}(\tau)e^{-\delta/\e}, \varphi_{\xi}+\tfrac12\xi\varphi\rangle\,.
\end{eqnarray*}
Then  
\begin{eqnarray*}
A^{\e}F_{1}^{\e}(\tau)&=&f''(\langle u^{\e}(\tau), \varphi\rangle)\langle u^{\e}(\tau)\bar{\eta}^{\e}(\tau), \varphi_{\xi}+\tfrac12\xi\varphi\rangle^{2} 
\\&&{}
+\sqrt{\e}f''(\langle u^{\e}(\tau), \varphi\rangle)[\langle u^{\e}(\tau), \mathcal{L}\varphi\rangle-\langle \tfrac12(u^{\e}(\tau))^{2}_{\xi}, \varphi\rangle]\langle (u^{\e}(\tau)\eta^{\e}(\tau))_{\xi}, \varphi  \rangle
\\&&{}
-\frac{1}{\sqrt{\e}}f'(\langle u^{\e}(\tau), \varphi\rangle)\langle (u^{\e}(\tau)\bar{\eta}^{\e}(\tau))_{\xi}, \varphi\rangle
\\&&{}
-f'(\langle u^{\e}, \varphi\rangle)\langle (u^{\e}(\tau)\bar{\eta}^{\e}(\tau))_{\xi}, (\varphi_{\xi}+\tfrac12\xi\varphi)\bar{\eta}^{\e}(\tau)\rangle
\\&&{}
-\sqrt{\e}f'(\langle u^{\e}(\tau), \varphi\rangle)\left[\langle u^{\e}(\tau), \mathcal{L}((\varphi_{\xi}+\tfrac12\xi\varphi)\bar{\eta}^{\e}(\tau))\rangle 
\right.\\&&\left.{}
-\langle\tfrac12(u^{\e}(\tau))^{2}_{\xi}, (\varphi_{\xi}+\tfrac12\xi\varphi)\bar{\eta}^{\e}(\tau)\rangle\right]\,.
\end{eqnarray*}
This completes the proof. 
\end{proof}
Now by the above construction of $(Y^{\e}, Z^{\e})$, 
\begin{equation*}
Y^{\e}(\tau)-f(\langle u^{\e}(\tau), \varphi\rangle)=-F_{1}^{\e}(\tau).
\end{equation*}
Then by Lemma~\ref{lem:F-1},  
\begin{equation*}
\lim_{\e\rightarrow 0}\mathbb{E}\sup_{0\leq \tau\leq T}|F_{1}^{\e}(\tau)|=0\,. 
\end{equation*}
Furthermore, by the $L^{\infty}(\R)$ estimates on~$u^{\e}(\tau)$, 
\begin{equation}
\sup_{\e}\mathbb{E}\|Z^{\e}(\tau)\|_{L^{2}(0, T)}<\infty\,.
\end{equation}
By Lemma~\ref{lem:f-tight}, we thus deduce the following theorem. 
\begin{theorem}
Assume Assumption~\ref{ass:a} holds, and $u^{\e}_{0}\in L^{2}(\Omega, L^{2}(K)\cap L^{\infty}(\R))$ satisfies~(\ref{e:ini-mM}), the family of distributions of  processes $\{u^{\e}\}_{0<\e\leq 1}$ is tight in the space~$C([0, \infty), L^{2}(K))$.
\end{theorem}
\begin{remark}\label{rem:stat-bound}
By~(\ref{e:u-bar-eps}), all discussions in this section and section~\ref{sec:estimates} hold for~$\bar{u}^{\e}$. So the above tightness result also holds for $\{\bar{u}^{\e}\}_{0<\e \leq 1}$\,.
\end{remark}

 
\subsection{Pass to the limit $\e\rightarrow 0$}
We show the convergence of $u^{\e}$ as $\e\rightarrow 0$ and determine the limit. For this we introduce a limit martingale problem and show any accumulation point of $\{u^{\e}\}$ is a solution to this martingale problem.  By the convergence result of Walsh~\cite[Theorem~6.15]{Walsh86}, we just need to consider finite dimensional distributions of~$\{\langle u^{\e}, \varphi_{1}\rangle, 
\ldots,\langle u^{\e}, \varphi_{n}\rangle \}$ 
for any $\varphi_{1}, \ldots, \varphi_{n}\in\mathcal{D}(\R)$;
that is, we just  pass to the limit $\e\rightarrow 0$ in the  following equality 
\begin{equation}
 \mathbb{E}\left\{\left[Y^{\e}(\tau)-Y^{\e}(s)-\int_{s}^{\tau} Z^{\e}(r)\,dr\right]h(\langle u^{\e}(r_{1}), \varphi_{1}\rangle, 
\ldots,\langle u^{\e}(r_{n}), \varphi_{n}\rangle) \right\}=0
\end{equation}
for any bounded continuous function $h$ and $0<r_{1}<\cdots<r_{n}<T$ with any $T>0$\,.
Denote by $u$ one limit point in the sense of distribution   of~$u^{\e}$ as $\e\rightarrow 0$ in the space~$C([0, \infty), L^{2}(K))$.  Notice that we can not  have the limit   $\lim_{\e\rightarrow 0}f(\langle u^{\e}(\tau), \varphi\rangle)=f(\langle u(\tau), \varphi\rangle)$ just with the convergence of~$u^{\e}$ to~$u$ in distribution. 
 However,  by the Skorohod theorem we construct new probability space and new variables  without changing distributions in~$C([0, \infty), L^{2}(K))$
(for simplicity we do not introduce new notation)   such that 
$u^{\e}$~almost surely converges to~$u$ in the space~$C([0, \infty); L^{2}(K))$. 
Then by the estimates in Lemma~\ref{lem:F-1} and the construction  of~$Y^{\e}$,  the limit of $Y^{\e}(\tau)-Y^{\e}(s)$ is $f(\langle u(\tau), \varphi \rangle)-f(\langle u(s), \varphi \rangle)$. 

Now we treat the integral term. First  denote by 
$Z_{k}^{\e}(\cdot)$, $k=1,2,3,4$\,, the first four terms of~$Z^{\e}(\cdot)$ and by $Z_{5}^{\e}(\cdot)$ the last two terms in~$Z^{\e}(\cdot)$. Then,  in distribution as $\e\rightarrow 0$\,, 
\begin{equation*}
\int_{s}^{\tau}Z_{1}^{\e}(r)\,dr\rightarrow \int_{s}^{\tau} f'(\langle u(r), \varphi \rangle)\langle u(r), \mathcal{L}\varphi\rangle\,dr\,,
\end{equation*}
and by the estimates on~$u^{\e}$ in section~\ref{sec:estimates} and estimates on~$\bar{\eta}^{\e}$ in section~\ref{sec:pre},
\begin{equation*}
\mathbb{E}\int_{s}^{\tau}|Z_{5}^{\e}(r)|dr\rightarrow 0\,.
\end{equation*}
 
Notice that $(u^{\e})^{2}$ is bounded in $L^{2}(0, T; L^{2}(K))$ for any $T>0$ and  by the tightness  of~$u^{\e}$ in the space~$C(0, T; L^{2}(K))$, $u^{\e}$~converges almost everywhere to~$u$ on~$[0, T]\times \R$, then  by the $L^{\infty}(\R)$~bound on~$u^{\e}$ and~$u$,  we have in distribution for any $T>0$
\begin{equation}\label{e:u-2}
\langle (u^{\e})^{2}, \varphi\rangle \rightarrow \langle u^{2},\varphi\rangle \quad \text{as } \e\rightarrow 0\,. 
\end{equation}
So in distribution as $\e\rightarrow 0$
\begin{equation*}
\int_{s}^{\tau}Z_{2}^{\e}(r)dr\rightarrow \int_{s}^{\tau} f'(\langle u(r), \varphi\rangle)\langle u(r)u_{\xi}(r), \varphi\rangle dr.
\end{equation*}
Next we treat terms~$Z_{3}^{\e}$ and~$Z_{4}^{\e}$. For any $u\in L^{2}(K) $ define the bilinear operator~$\Sigma(u) $ such that
\begin{equation}
\langle \Sigma(u)\varphi, \varphi\rangle= \int_{\R}\int_{\R}u(\xi)u(\zeta)q(\xi, \zeta)(\varphi(\xi)K(\xi))_{\xi}(\varphi(\zeta)K(\zeta))_{\zeta}\,d\xi \, d\zeta
\end{equation}
for any $\varphi\in \mathcal{D}(\R)$, and define the linear operator 
\begin{eqnarray}
\langle A(u), \varphi\rangle &=&\tfrac12\int_{\R}u(\xi)q(\xi, \xi)(\varphi(\xi) K(\xi))_{\xi\xi}\,d\xi\\&&{}+\tfrac12\int_{\R}u(\xi)q'(\xi, \xi)(\varphi(\xi)K(\xi))_{\xi}\,d\xi\,.\nonumber
\end{eqnarray} 
For this operator~$\Sigma$ the following lemma holds. 
\begin{lemma}\label{lem:uW}
For any $u\in H^{1}(K)$, let $B(\tau, \xi)=(u(\xi)W(\tau, \xi))_{\xi}$\,, then $B$ is an $L^{2}(K)$~valued Wiener process with the covariation operator~$\Sigma(u)$.
\end{lemma}
\begin{proof}
The proof is direct.  By~(\ref{e:W}), 
\begin{equation*}
W_{\xi}=\sum_{k=1}^{\infty}\sqrt{q_{k}}e'_{k}(\xi)\beta_{k}(\tau).
\end{equation*}
Then by the representation of $q(\xi, \zeta)$,
\begin{eqnarray*}
&&\mathbb{E}B(\tau,\xi)B(\tau, \zeta)
\\&=&\mathbb{E}\left(u(\xi)\sum_{k}\sqrt{q_{k}}e_{k}(\xi)\beta_{k}(\tau)\right)_{\xi}
\left(u(\zeta)\sum_{l}\sqrt{q_{l}}e_{l}(\zeta)\beta_{l}(\tau)\right)_{\xi}\\
&=&u_{\xi}(\xi)u(\zeta)q_{\zeta}(\xi, \zeta)+u_{\xi}(\xi)u_{\zeta}(\zeta)q(\xi, \zeta)+u(\xi)u_{\zeta}(\zeta)q_{\xi}(\xi, \zeta)\\&&{}+u(\xi)u(\zeta)q_{\xi,\zeta}(\xi, \zeta).
\end{eqnarray*}
By the definition of~$\Sigma(u)$ this proves the lemma.
\end{proof}

To pass to the limit $\e\rightarrow 0$ in $Z^{\e}_{3}$~and~$Z_{4}^{\e}$ we apply again the perturbation method~\cite{Kushner}. Let 
\begin{eqnarray*}
F_{3}^{\e}(\tau)&=&f''(\langle u^{\e}(\tau), \varphi\rangle) \int_{\tau}^{\infty}  \mathbb{E} \Big[ \langle u^{\e}(\tau)\bar{\eta}^{\e}(s), \varphi_{\xi}+\tfrac12\xi\varphi\rangle^{2}\\&&{}-\langle \Sigma(u^{\e}(\tau))\varphi, \varphi\rangle\,\big|\,\mathcal{F}_{\tau}^{\e}\Big]\,ds
\end{eqnarray*}
and 
\begin{eqnarray*}
F_{4}^{\e}(\tau)&=&f'(\langle u^{\e}(\tau), \varphi\rangle)\int_{\tau}^{\infty}\mathbb{E}\Big[\langle (u^{\e}(\tau)\bar{\eta}^{\e}(s))_{\xi}, (\varphi_{\xi}+\tfrac12\xi\varphi)\bar{\eta}^{\e}(s)\rangle\\&&{}-\langle A(u^{\e}(\tau)), \varphi\rangle\,\big|\,\mathcal{F}_{\tau}^{\e}\Big]ds\,.
\end{eqnarray*}
By the construction of $\bar{\eta}^{\e}$ and $\bar{\eta}_{\xi}^{\e}$\,,
\begin{equation}\label{e:eta-eta}
\mathbb{E}\left[\bar{\eta}^{\e}(s, \xi)\bar{\eta}^{\e}(s, \zeta) \mid \mathcal{F}_{\tau}^{\e}\right]=e^{-2({s-\tau})/{\e}}\bar{\eta}^{\e}(\tau,\xi)\bar{\eta}^{\e}(\tau,\zeta)+\tfrac12q(\xi, \zeta)(1-e^{-2({s-\tau})/{\e}}),
\end{equation}
\begin{equation}
\mathbb{E}\left[\bar{\eta}^{\e}(s, \xi)\bar{\eta}_{\xi}^{\e}(s, \xi) \mid \mathcal{F}_{\tau}^{\e}\right]=e^{-2({s-\tau})/{\e}}\bar{\eta}^{\e}(\tau,\xi)\bar{\eta}^{\e}_{\xi}(\tau,\xi)+\tfrac12q'(\xi, \xi)(1-e^{-2({s-\tau})/{\e}}).\label{e:eta-eta-xi}
\end{equation}
Then   
\begin{eqnarray}\label{e:F-3}
F_{3}^{\e}(\tau)= \tfrac{\e}{2} f''(\langle u^{\e}(\tau),\varphi\rangle)\big[\langle u^{\e}(\tau)\bar{\eta}^{\e}(\tau),\varphi_{\xi}+\tfrac12\xi\varphi\rangle^{2}
-\tfrac12\langle \Sigma(u^{\e}(\tau))\varphi, \varphi\rangle\big]
\end{eqnarray}
and  
\begin{equation}
F_{4}^{\e}(\tau)=\tfrac{\e}{2} f'(\langle u^{\e}(\tau), \varphi\rangle)\big[\langle (u^{\e}(\tau)\bar{\eta}^{\e}(\tau))_{\xi}, (\varphi_{\xi}+\tfrac12\xi\varphi)\bar{\eta}^{\e}(\tau)\rangle
-\tfrac12\langle A(u^{\e}(\tau)), \varphi\rangle \big]\,.  
\end{equation} 
Then by the estimates on $u^{\e}(\tau)$, $\bar{\eta}^{\e}(\tau)$ and $\bar{\eta}^{\e}_{\xi}(\tau)$, direct computation  yields this lemma.
\begin{lemma}
As $\e\rightarrow 0$\,,
\begin{equation*}
\sup_{\tau\geq 0}\mathbb{E}F_{3}^{\e}(\tau)=\mathcal{O}(\e), \quad \text{and} \quad \sup_{\tau\geq 0}\mathbb{E}F_{4}^{\e}(\tau)=\mathcal{O}(\e).
\end{equation*}
\end{lemma}

 Now following the same discussion as in Lemma~\ref{lem:Z} and~(\ref{e:eta-eta})--(\ref{e:eta-eta-xi}), we have the following lemma.
 \begin{lemma}
\begin{eqnarray*}
A^{\e}F_{3}^{\e}(\tau)&=& f''(\langle u^{\e}(\tau), \varphi)[\langle \tfrac12\Sigma(u^{\e}(\tau))\varphi,\varphi\rangle\\&&{}-\langle u^{\e}(\tau)\bar{\eta}^{\e}(\tau), \varphi_{\xi}+\tfrac12\xi\varphi\rangle]+R_{3}^{\e}(\tau)
\end{eqnarray*}
and 
\begin{eqnarray*}
A^{\e}F_{4}^{\e}(\tau)&=& f'(\langle u^{\e}(\tau), \varphi\rangle)[\langle Au^{\e}(\tau), \varphi\rangle\\&&{}-\langle (u^{\e}(\tau)\bar{\eta}^{\e}(\tau))_{\xi}, (\varphi_{\xi}+\tfrac12\xi\varphi)\bar{\eta}^{\e}(\tau)]+R_{4}^{\e}(\tau)
\end{eqnarray*}
 with 
 \begin{equation*}
 \sup_{\tau\geq 0}\mathbb{E}|R^{\e}_{3}(\tau)|=\mathcal{O}(\e)\quad \text{and}\quad \sup_{\tau\geq 0}\mathbb{E}|R^{\e}_{4}(\tau)|=\mathcal{O}(\e)
 \end{equation*}
as $\e\rightarrow 0$\,. 
 \end{lemma}
\begin{proof}
This is similar to the discussion in the proof of Lemma~\ref{lem:Z}. 
First we have  for any $\delta>0$  
\begin{eqnarray*}
&&\mathbb{E}[F^{\e}_{3}(\tau+\delta) \mid \mathcal{F}_{\tau}^{\e}]\\
&=&\tfrac{\e}{2}\mathbb{E}\Big[\big(f''(\langle u^{\e}(\tau+\delta), \varphi\rangle)-f''(\langle u^{\e}(\tau), \varphi\rangle)\big)\\
&&{}\times \big(\langle u^{\e}(\tau+\delta)\bar{\eta}^\e(\tau+\delta), \varphi_{\xi}+\tfrac12\xi\varphi\rangle^{2}  \\
&&{}\quad-\langle \Sigma(u^{\e}(\tau+\delta))\varphi, \varphi\rangle \big)  \mid \mathcal{F}_{\tau}^{\e}\Big]\\
&&{}+\tfrac{\e}{2}f''(\langle u^{\e}(\tau), \varphi\rangle)\mathbb{E}\Big[ \langle (u^{\e}(\tau+\delta)-u^{\e}(\tau))\bar{\eta}^{\e}(\tau+\delta), \varphi_{\xi}+\tfrac12\xi\varphi\rangle^{2}\\
&&{}\quad -\langle (\Sigma(u^{\e}(\tau+\delta))-\Sigma(u^{\e}(\tau)))\varphi, \varphi\rangle \mid \mathcal{F}_{\tau}^\e \Big]\\&&{}+\tfrac{\e}{2}f''(\langle u^{\e}(\tau), \varphi\rangle)\mathbb{E}\Big[ \langle u^{\e}(\tau)\bar{\eta}^{\e}(\tau+\delta), \varphi_{\xi}+\tfrac12\xi\varphi\rangle^{2}-\langle\Sigma(u^{\e}(\tau))\varphi, \varphi\rangle  \mid \mathcal{F}_{\tau}^{\e} \Big]
\end{eqnarray*}
 and 
 \begin{eqnarray*}
&&\mathbb{E}[F_{4}^{\e}(\tau+\delta) \mid \mathcal{F}_{\tau}^{\e}]\\
&=&\tfrac{\e}{2}\mathbb{E}\Big[\big(f'(\langle u^{\e}(\tau+\delta), \varphi\rangle)-f'(\langle u^{\e}(\tau), \varphi\rangle)\big)\\
&&{}\times \big( \langle (u^{\e}(\tau+\delta)\bar{\eta}^{\e}(\tau+\delta))_{\xi}, (\varphi_{\xi}+\tfrac12\xi\varphi)\bar{\eta}^{\e}(\tau+\delta)\rangle \\
&&\quad -\langle A(u^{\e}(\tau+\delta)), \varphi\rangle\big) \mid \mathcal{F}_{\tau}^{\e} \Big]\\&&{}+\tfrac{\e}{2}f'(\langle u^{\e}(\tau), \varphi\rangle)\mathbb{E}\Big[\big\langle [(u^{\e}(\tau+\delta)-u^{\e}(\tau))\bar{\eta}^{\e}(\tau+\delta)]_{\xi}, \\&&\quad (\varphi_{\xi}+\tfrac12\xi\varphi_{\xi})\bar{\eta}^{\e}(\tau+\delta)\big\rangle-\langle A(u^{\e}(\tau+\delta))-A(u^{\e}(\tau)), \varphi\rangle \mid \mathcal{F}_{\tau}^{\e} \Big]\\
&&{}+\tfrac{\e}{2}f'(\langle u^{\e}(\tau), \varphi\rangle )\mathbb{E}\Big[ \langle (u^{\e}(\tau)\bar{\eta}^{\e}(\tau+\delta))_{\xi}, (\varphi_{\xi}+\tfrac12\xi\varphi)\bar{\eta}^{\e}(\tau+\delta)\rangle\\&&\quad -\langle A(u^{\e}(\tau)), \varphi\rangle \mid \mathcal{F}_{\tau}^{\e}\Big]\,.
\end{eqnarray*}
Then by the definition of $A^{\e}$ and~(\ref{e:eta-eta})--(\ref{e:eta-eta-xi}),  direct computation yields the result.
The proof is complete.
\end{proof}
Now we have the following $\mathcal{F}_{\tau}^{\e}$ martingale 
\begin{eqnarray}
\mathcal{M}^{\e}(\tau)&=&f(\langle u^{\e}(\tau), \varphi\rangle)- f(\langle u^{\e}_{0}, \varphi\rangle)-    F_{1}^{\e}(\tau)-F_{3}^{\e}(\tau)-F_{4}^{\e}(\tau)
\nonumber\\&&{}
-\int_{0}^{\tau}f'(\langle u^{\e}(r), \varphi\rangle)\Big[\langle u^{\e}(r), \mathcal{L}\varphi\rangle+\left\langle\tfrac12(u^{\e}(r))^{2}, \varphi_{\xi}+\tfrac12\xi\varphi\right\rangle
\nonumber \\ &&\hspace{10em}{}
+\langle A(u^{\e}(r)), \varphi\rangle \Big]dr\nonumber\\&&{}
-\tfrac12\int_{0}^{\tau}f''(\langle u^{\e}(r), \varphi\rangle)\langle \Sigma(u^{\e}(r))\varphi, \varphi\rangle dr+R^{\e}(\tau)\nonumber
\end{eqnarray}
where 
\begin{equation*}
R^{\e}(\tau)=\int_{0}^{\tau}[Z_{5}^{\e}(s)+R_{3}^{\e}(s)+R_{4}^{\e}(s)]\,ds 
\end{equation*}
with $\mathbb{E}|R^{\e}(\tau)|=\mathcal{O}(\e)$ as $\e\rightarrow 0$\,.
Now passing to the limit $\e\rightarrow 0$\,,  the distribution of the limit~$u$ solves the martingale problem
\begin{eqnarray}
\mathcal{M}(\tau)&=&f(\langle u(\tau), \varphi\rangle)- f(\langle u_{0}, \varphi\rangle) 
-\int_{0}^{\tau}f'(\langle u(r), \varphi\rangle)\\&&{}\times\Big[\langle u(r), \mathcal{L}\varphi\rangle+\left\langle\tfrac12(u(r))^{2}, \varphi_{\xi}+\tfrac12\xi\varphi\right\rangle+\langle A(u(r)), \varphi\rangle \Big]dr\nonumber\\&&{}
-\tfrac12\int_{0}^{\tau}f''(\langle u(r), \varphi\rangle)\langle \Sigma(u(r))\varphi, \varphi\rangle dr\nonumber
\end{eqnarray}
which, by Lemma~\ref{lem:uW}, is equivalent~\cite{MM88} to the martingale solution to the \textsc{spde}~(\ref{e:SPDE}) 
for some new Wiener process~$\bar{W}$ with the same distribution as that of~$W$.
By the general  theory of \textsc{spde}s~\cite{Prato}, the martingale solution to the \spde~(\ref{e:SPDE}) is unique in the space $L^{2}(0, T; H^{1}(K)) \cap C(0, T; L^{2}(K))$ for any $T>0$\,. Then we deduce the following theorem.
\begin{theorem}
Assume Assumption~\ref{ass:a} holds and that the initial data $u^{\e}_{0}\in L^{2}(\Omega, L^{2}(K)\cap L^{\infty}(\R))$ satisfies~(\ref{e:ini-mM}).  Moreover, $u^{\e}_{0}$~converges in distribution to~$u_{0}$.   The solution of  \textsc{rde}~(\ref{e:rBurgers}),~$u^{\e}$, converges in distribution in the space~$C([0, \infty), L^{2}(K))$ to~$u$ which solves the \textsc{spde}~(\ref{e:SPDE}) with initial data~$u_{0}$.
\end{theorem}

Notice that the \textsc{spde}~(\ref{e:SPDE}) is different from the stochastic Burgers' equation~(\ref{e:tau-xi-Burgers}). Now  we consider random Burgers' type equation~(\ref{e:rBurgers2}). 
By the assumption that $q(\xi)\in H^{2}(K)$ and~(\ref{e:q-bound}),  the extra terms $\tfrac12(u^{\e}q)_{\xi\xi}$ and $\tfrac12(u^{\e}q')_{\xi}$   do not change the estimates in sections~\ref{sec:self-similar-sol}--\ref{sec:diffusionApp}. 
Then we derive this theorem. 
\begin{theorem}\label{thm:converg}
Assume Assumption~\ref{ass:a} holds and $u_{0}^{\e}\in L^{2}(\Omega, L^{2}(K)\cap L^{\infty}(\R))$ converges in distribution to $u_{0}\in L^{2}(\Omega, L^{2}(K)\cap L^{\infty}(\R))$ as $\e \rightarrow 0$\,. Moreover, assume~(\ref{e:ini-mM}) holds. 
The solution  of  \textsc{rde}~(\ref{e:rBurgers2})  converges, as $\e\rightarrow 0$\,,  in distribution in the space~$C([0, \infty), L^{2}(K))$ to the  solution of \textsc{spde}~(\ref{e:tau-xi-Burgers}) with initial data~$u_{0}$.
\end{theorem}
By the above theorem and Remark~\ref{rem:stat-bound}
we have the following result on the convergence of stationary statistical solutions. 
Denote by $\bar{\mathfrak{P}}^{\e}=\mathcal{D}(\bar{\hat{u}}^{\e}, \bar{\hat{\eta}}^{\e})$, a stationary statistical solution to the system~(\ref{e:rBurgers2}) coupled with $\bar{\eta}^{\e}$\,.  Let $\bar{\mathbb{P}}^{\e}=\mathcal{D}(\bar{\hat{u}}^{\e})$\,. Then we have the following corollary.
\begin{coro}\label{cor:con}
For $\e\rightarrow 0$\,, there is a sequence $\e_{n}\rightarrow 0$ as $n\rightarrow\infty$\,, such that 
\begin{equation*}
\bar{\mathbb{P}}^{\e_{n}}\rightarrow \bar{\mathbb{P}}\quad \text{weakly as } n\rightarrow\infty\,,
\end{equation*}
where $\bar{\mathbb{P}}$ is a probability space on $C([0, \infty); L^{2}(K))$\,, which is a stationary statistical solution to stochastic Burgers equation~(\ref{e:tau-xi-Burgers})
\end{coro}

We have shown that for any $\mathfrak{y}=\mu*\bar{\nu}\in\mathfrak{M}_{2}$\,, there is a stationary measure $\bar{\mathfrak{y}}^{\e}=\bar{\mu}^{\e}*\bar{\nu}$ such that $\mathcal{P}_{\tau}^{\e}\mathfrak{y}$ weakly converges to $\bar{\mathfrak{y}}^{\e}$ as $\tau\rightarrow\infty$\,. 
Denote by $\bar{u}^{\e}$ the solution to~(\ref{e:rBurgers2})   with initial data distributes as $\bar{\mu}^{\e}$\,.  
Then the probability measure $\bar{\mathfrak{P}}^{\e}=\mathcal{D}(\bar{u}^{\e}, \eta^{\e})$ is a stationary statistical solution to the system~(\ref{e:rBurgers2}) coupled with~(\ref{e:eta}).
We are concerned with the marginal distribution $\bar{\mathbb{P}}^{\e}=\mathcal{D}(\bar{u}^{\e})$\,.

For any statistical solution $\mathfrak{P}^{\e}$ to~(\ref{e:rBurgers2}) coupled with~(\ref{e:eta}), denote  by~$\mathbb{P}^{\e}$ the marginal distribution,  then $\mathbb{P}_{\tau}^{\e}$~converges weakly to~$\bar{\mathbb{P}}^{\e}$ as $\tau\rightarrow\infty$. 
Let $\mathfrak{P}^{\e}=\mathcal{D}(\hat{u}^{\e}, \hat{\eta}^{\e})$ with $(\hat{u}^{\e}, \hat{\eta}^{\e})$
solving~(\ref{e:rBurgers2}) coupled with~(\ref{e:eta})  with new Wiener process~$\hat{W}$ distributing the same as~$W$. 
Then $\mathbb{P}_{\tau}^{\e}=\mathcal{D}(\hat{u}^{\e}(\cdot+\tau))$\,.
 By the continuous dependence on initial  data of the solution,   as
 \begin{equation*}
 \mathcal{D}(\hat{u}^{\e}(\tau))\rightarrow \mathcal{D}(\bar{u}^{\e}(0))\quad \text{weakly as } \tau\rightarrow\infty\,,
 \end{equation*}
 we have 
\begin{equation*}
\mathcal{D}(\hat{u}^{\e}(\cdot+\tau))\rightarrow \mathcal{D}(\bar{u}^{\e})\quad \text{weakly as } \tau\rightarrow\infty\,.
\end{equation*}
 That is $\mathbb{P}^{\e}_{\tau}\rightarrow \bar{\mathbb{P}}^{\e}$ weakly as $\tau\rightarrow\infty$\,.  
 Furthermore by the convergence result Corollary~\ref{cor:con}, there is   $\e_{n}\rightarrow0$ as $n\rightarrow \infty$\,, $\bar{\mathbb{P}}^{\e_{n}}$ converges weakly to some probability measure~$\bar{\mathbb{P}}$ which is a stationary statistical solution of \spde{}~(\ref{e:tau-xi-Burgers}) in the space $C([0, \infty), L^{2}(K))$.

 Now consider any solution $u\in C([0, \infty); L^{2}(K))$ to the \spde{}~(\ref{e:tau-xi-Burgers}). 
 Then $\mathbb{P}=\mathcal{D}(u)$ is a statistical solution  to the \spde{}~(\ref{e:tau-xi-Burgers}) and, by Theorem~\ref{thm:converg}, there is a statistical solution $\mathfrak{P}^{\e}$ to the system~(\ref{e:rBurgers2}) coupled with $\eta^{\e}$,  such that the marginal distribution~$\mathbb{P}^{\e}$ converges to~$\mathbb{P}$ weakly as $\e\rightarrow 0$\,. 
 Moreover $\mathbb{P}^{\e}_{\tau}$ converges weakly to $\mathbb{P}_{\tau}$ as $\e\rightarrow 0$\,.  One can choose $u_{0}^{\e}=u_{0}$ in Theorem~\ref{thm:converg}. 
 
Then for any $\delta>0$\,, there is $\mathcal{T}>0$ such that for $\tau>\mathcal{T}$ and for sufficiently large~$n$  
\begin{equation*}
d_{\text{P}}(\mathbb{P}_{\tau}, \bar{\mathbb{P}})\leq 
d_{\text{P}}(\mathbb{P}_{\tau}, \mathbb{P}^{\e_{n}}_{\tau})+d_{\text{P}}(\mathbb{P}_{\tau}^{\e_{n}}, \bar{\mathbb{P}}^{\e_{n}})+d_{\text{P}}(\bar{\mathbb{P}}^{\e_{n}}, \bar{\mathbb{P}})\leq \delta\,.
\end{equation*} 
That is $d_{\text{P}}(\mathbb{P}_{\tau}, \bar{\mathbb{P}})\rightarrow 0$ as $\tau\rightarrow\infty$\,.  
Assume $\mathbb{P}=\mathcal{D}(\hat{u})$ and $\bar{\mathbb{P}}=\mathcal{D}(\bar{u})$\,, then $\bar{u}$~is a stationary process and 
\begin{equation*}
\mathcal{D}(\hat{u}(\tau))\rightarrow \mathcal{D}(\bar{u}(0))\quad \text{weakly as } \tau\rightarrow\infty\,.
\end{equation*}
Then we conclude with the following main theorem.
\begin{theorem}\label{thm:main}
Assume Assumption~\ref{ass:a} holds.  For any solution of the stochastic Burgers' equation~(\ref{e:tx-Burgers}) with  $w(1,x)\in L^{2}(K)\cap L^{\infty}(\R)$, there is a unique self-similar solution~$\bar{w}(t,x)$ such that 
\begin{equation*}
\sqrt{t}w(t,x)-\sqrt{t}\bar{w}(t,x)\rightarrow 0 \quad \text{as } t\rightarrow\infty \quad \text{in distribution in } L^{2}(\R).
\end{equation*}
\end{theorem}


\appendix

\section{Existence result for stochastic Burgers' type  equation}\label{apd:1}  
We prove Theorem~\ref{thm:wellpose},  Noticing  that the nonlinearity is non-Lipschitz,  for any $R>0$ we introduce the  following smooth cut-off function $\chi_{R}: [0, \infty)\rightarrow\R$ defined by 
\begin{equation*}
\chi_{R}(x)=
\begin{cases}
1\,,& 0\leq x\leq R\,,\\
0\,,& x\geq 2R\,.
\end{cases}
\end{equation*}
Consider the following system 
\begin{equation}\label{e:cut-sys}
du^{R}=[\mathcal{L}u^{R}-\chi_{R}(\|u^{R}\|_{H^{1}(K)})u^{R}u^{R}_{\xi}]\,d\tau+(u^{R}dW)_{\xi}\,,\quad u^{R}(0)=u_{0}\,.
\end{equation}
Denote by $B_{R}(u^{R}):=B_{R}(u^{R}, u^{R}):=\chi_{R}(\|u^{R}\|_{H^{1}(K)})u^{R}u^{R}_{\xi}$\,, then $B_{R}: H^{1}(K)\rightarrow L^{2}(K)$ is Lipschtiz in the following sense 
\begin{equation*}
\|B_{R}(u)-B_{R}(v)\|_{L^{2}(K)}\leq L_{R}\|u-v\|_{H^{1}(K)} 
\end{equation*}
for $u,v\in H^{1}(K)$.  Now for each $u_{0}\in L^{2}(K)$ define the nonlinear operator 
\begin{equation*}
\mathcal{T}(u^{R})(\tau)=S(\tau)u_{0}+\int_{0}^\tau S(\tau-s)B_{R}(u^{R}(s))\,ds+\int_{0}^{\tau}S(\tau-s)(u^{R}(s)dW(s))_{\xi}\,.
\end{equation*}
Then $\mathcal{T}$ maps $L^{2}(\Omega, C(0, T; L^{2}(K))\cap L^{2}(0, T; H^{1}(K)))$ to itself. 
For any $u^{R}\in L^{2}(\Omega, C(0, T; L^{2}(K))\cap L^{2}(0, T; H^{1}(K)))  $, by the  properties of the semigroup $S(\tau)$,
\begin{eqnarray*}
&&\|(\mathcal{T}u^{R})(\tau)\|_{L^{2}(K)}\\&\leq&  \|u_{0}\|_{L^{2}(K)}
+\int_{0}^{\tau} \|u^{R}u^{R}_{\xi}\|_{L^{2}(K)}\,ds+\left\|\int_{0}^{\tau}S(\tau-s)(u^{R}dW)_{\xi}\right\|_{L^{2}(K)}\\
&\leq&\|u_{0}\|_{L^{2}(K)}
+\int_{0}^{\tau} \|u^{R}\|_{L^{\infty}(\R)}\|u^{R}_{\xi}\|_{L^{2}(K)}\,ds
\\&&{}
+2\left\|\int_{0}^{\tau}S(\tau-s)(u^{R}dW)_{\xi}\right\|^{2}_{L^{2}(K)}+2\\
&\leq&\|u_{0}\|_{L^{2}(K)}
+\int_{0}^{\tau} \|u^{R}\|^{2}_{H^{1}(K)}\,ds+2\left\|\int_{0}^{\tau}S(\tau-s)(u^{R}dW)_{\xi}\right\|^{2}_{L^{2}(K)}+2
\end{eqnarray*} 
For the third term,  by  Assumption~\ref{ass:a} and the properties of the stochastic integral~\cite[Lemma 7.2]{Prato},  for some constant $C>0$
\begin{eqnarray*}
&&\mathbb{E}\max_{0\leq \tau\leq T}\left\|\int_{0}^{\tau}S(\tau-s)(u^{R}dW)_{\xi} \right\|^{2}_{L^{2}(K)}\\&\leq&  2\mathbb{E}\max_{0\leq \tau\leq T}\left\|\int_{0}^{\tau}S(\tau-s)u^{R}_{\xi}dW \right\|^{2}_{L^{2}(K)}+2\mathbb{E}\max_{0\leq \tau\leq T}\left\|\int_{0}^{\tau}S(\tau-s)u^{R}dW_{\xi} \right\|^{2}_{L^{2}(K)}\\
&\leq& C\int_{0}^{T}\|u^{R}(s)\|^{2}_{H^{1}(K)}\,ds+C\int_{0}^{T}\|u^{R}(s)\|^{2}_{L^{2}(K)}\,ds\,.
\end{eqnarray*}
Then for some constant $C>0$
\begin{equation*}
\mathbb{E}\|\mathcal{T}u^{R}\|^{2}_{C(0, T; L^{2}(K))}\leq \|u_{0}\|^{2}_{L^{2}(K)}+C\int_{0}^{T}\|u^{R}(s)\|^{2}_{H^{1}(K)}\,ds\,.
\end{equation*}
On the other hand 
\begin{eqnarray*}
&&\|(\mathcal{T}u^{R})(\tau)\|_{H^{1}(K)}\\&\leq& \|S(\tau) u_{0}\|_{H^{1}(K)}+\left\|\int_{0}^{\tau}S(\tau-s)\chi_{R}(\|u^{R}\|_{H^{1}(K)})u^{R}u^{R}_{\xi}\,ds\right\|_{H^{1}(K)}\\&&{}+\left\|\int_{0}^{\tau}S(\tau-s)(u^{R}dW)_{\xi}\right\|_{H^{1}(K)}\,.
\end{eqnarray*}
By the properties of the semigroup $S(\tau)$, for $T>0$
\begin{equation*}
\int_{0}^{T}\|S(\tau)u_{0}\|^{2}_{H^{1}(K)}\,d\tau\leq \frac{1}{2}\|u_{0}\|^{2}_{L^{2}(K)}\,,
\end{equation*}
and for some constant $C_{R}>0$
\begin{eqnarray*}
&&\int_{0}^{T}\left\|\int_{0}^{\tau}S(\tau-s)\chi_{R}(\|u^{R}\|_{H^{1}(K)})u^{R}u^{R}_{\xi}\,ds\right\|^{2}_{H^{1}(K)}\,d\tau\\
&\leq& C_{R}\int_{0}^{T}\|u^{R}(\tau)\|^{2}_{H^{1}(K)}\,d\tau\,.
\end{eqnarray*}
 For the stochastic term, first we have  
\begin{eqnarray*}
&& \left\|\int_{0}^{\tau}S(\tau-s)(udW)_{\xi}\right\|_{H^{1}(K)}\\&\leq &
 \left\|\int_{0}^{\tau}S(\tau-s)u_{\xi}dW\right\|_{H^{1}(K)}
 + \left\|\int_{0}^{\tau}S(\tau-s)u(dW)_{\xi}\right\|_{H^{1}(K)}\,. 
\end{eqnarray*}
By the properties of~$S(\tau)$, and direct calculation by the expanding $u_{\xi}^{R}$ and~$W$ in terms of the orthonormal basis, for some constant $C>0$
\begin{equation*}
\mathbb{E}\int_{0}^{T}\left\| \int_{0}^{\tau} S(\tau-s)u^{R}_{\xi}dW(s)\right\|^{2}_{H^{1}(K)}\,d\tau\leq C\int_{0}^{T}\|u^{R}\|^{2}_{H^{1}(K)}\,d\tau\,.
\end{equation*}
Similarly 
\begin{equation*}
\mathbb{E}\int_{0}^{T}\left\| \int_{0}^{\tau} S(\tau-s)u^{R}dW_{\xi}(s)\right\|^{2}_{H^{1}(K)}\,d\tau\leq C\int_{0}^{T}\|u^{R}\|^{2}_{H^{1}(K)}\,d\tau\,.
\end{equation*}
Then for some constants $C>0$ and $C_{R}>0$\,, such that 
\begin{equation*}
\mathbb{E}\int_{0}^{T}\|(\mathcal{T}u^{R})(\tau)\|^{2}_{H^{1}(K)}\,d\tau\leq C\|u_{0}\|^{2}_{L^{2}(K)}+C_{R}\int_{0}^{T}\|u^{R}(\tau)\|^{2}_{H^{1}(K)}\,d\tau\,.
\end{equation*}
 That is, $\mathcal{T}$ maps $L^{2}(\Omega, C(0, T; L^{2}(K))\cap L^{2}(0, T; H^{1}(K)))$ to itself.  Then by the classical method for  the wellposedness of \spde{}s~\cite{Prato} a  unique solution~$u^{R}$ exists for the cut-off system~(\ref{e:cut-sys}) in $L^{2}(\Omega, C(0, T; L^{2}(K))\cap L^{2}(0, T; H^{1}(K)))$. Now define an increasing sequence of stopping times~$\{\tau_{n}\}$:
\begin{equation*}
\tau_{R}:=\inf\{\tau>0: \|u^{R}(\tau)\|_{H^{1}(K)}\geq R \} 
\end{equation*}
when it exits, and $\tau_{R}=\infty$ otherwise. Let $\tau_{\infty}=\lim_{R\rightarrow\infty}\tau_{R}$\,, a.s., and set $u^{\tau_{R}}(\tau)=u^{R}(\tau \wedge\tau_{R})$. 
Then $u^{\tau_{R}}$ is a local solution of equation~(\ref{e:tau-xi-Burgers}) for $0\leq \tau\leq \tau_{R}$\,.
To show the limit $R\rightarrow\infty$ of $u^{\tau_{R}}$ we need an estimate on~$\|u\|_{H^{1}(K)}$.   Notice  the functional~$\|\cdot\|^{2}_{H^{1}(K)}$ is continuous and  differentiable on~$H^{1}(K)$.  Applying It\^o's formula~\cite{Chow} to $\|u_{s}(\tau)\|^{2}_{H^{1}(K)}$  yields 
\begin{eqnarray*}
&&\frac12\frac{d}{dt}\|u_{s}\|^{2}_{H^{1}(K)}\\&=&\langle \mathcal{L} u_{s}, -\mathcal{L}u_{s} \rangle-\langle uu_{\xi}, \mathcal{L}u \rangle+\langle (u\dot{W})_{\xi}, \mathcal{L}u \rangle+\frac12\left[\|\mathcal{L}^{1/2}u_{\xi}\|^{2}_{\mathcal{L}_{2}^{Q}}+\|\mathcal{L}^{1/2}u\|^{2}_{\mathcal{L}_{2}^{Q'}} \right].
\end{eqnarray*}
By the Cauchy inequality, for some small $\varepsilon>0$\,, there is $C_{\varepsilon}>0$ such that 
\begin{equation*}
\langle uu_{\xi}, \mathcal{L}u  \rangle\leq \varepsilon \|u\|^{2}_{H^{2}(K)}+C_{\varepsilon}\|u\|^{2}_{H^{1}(K)}\,.
\end{equation*}
Then, by the assumption~(\ref{e:q-bound}) on $q$ and~$q'$,  
\begin{eqnarray*}
&&\frac12\frac{d}{dt}\mathbb{E}\|u_{s}\|^{2}_{H^{1}(K)}\\&\leq& -c_{\varepsilon}\mathbb{E}\|u_{s}\|^{2}_{H^{2}(K)}+c_{\varepsilon}\mathbb{E}\|u_{c}\|^{2}_{H^{2}(K)}+C'_{\varepsilon}\mathbb{E}\|u_{s}\|^{2}_{H^{1}(K)}+C'_{\varepsilon}\mathbb{E}\|u_{c}\|^{2}_{H^{1}(K)}
\end{eqnarray*}
for some constants $c_{\varepsilon}>0$ and $C_{\varepsilon}>0$\,. Then by Gronwall lemma, 
\begin{equation*}
\mathbb{E}\|u(\tau)\|^{2}_{H^{1}(K)}\leq K(\tau)
\end{equation*}
for some function $K: (0, \infty)\rightarrow [0, +\infty)$.

Now we show $u(\tau):=\lim_{R\rightarrow \infty}u^{R}$ is the global solution to~(\ref{e:tau-xi-Burgers}). By the above estimates, for $\tau\in [0, T\wedge\tau_{R}]$
\begin{equation*}
\mathbb{E}\|u^{\tau_{R}}(\tau)\|_{H^{1}(K)}\leq K(\tau)\,.
\end{equation*}
Also 
\begin{eqnarray*}
\mathbb{E}\|u^{\tau_{R}}(T)\|^{2}_{H^{1}(K)}&=&\mathbb{E}\|u^{R}(T\wedge\tau_{R})\|_{H^{1}(K)}^{2}\\
&\geq&\mathbb{E}\left\{ 1_{\{\tau_{R}\leq T\}}\|u^{R}(T\wedge\tau_{R})\|_{H^{1}(K)}^{2} \right\}\\
&\geq&\mathbb{P}\{\tau_{R}\leq T\}R^{2}\,.
\end{eqnarray*}
Then 
\begin{equation*}
\mathbb{P}\{\tau_{R}\leq T\}\leq \frac{K(T)}{R^{2}}
\end{equation*}
which, by the Borel--Cantelli Lemma, implies 
\begin{equation*}
\mathbb{P}\left\{ \tau_{\infty}>T \right\}=1\,.
\end{equation*}
Then we have the global well-posedness of the equation~(\ref{e:tau-xi-Burgers}). 
By the property of~$S(\tau)$, this mild solution is also a weak solution~\cite{Prato},


\section{Statistical  solution}\label{apd:2}

We introduce the statistical  solution of the  random  Burgers equation~(\ref{e:rBurgers}) coupled with~$\eta^{\e}$. 
The statistical solution was introduced to study universal properties of turbulent flows~\cite[e.g.]{Foi72, Foi73, VF88},

The system of random equations~(\ref{e:rBurgers2}) coupled with~(\ref{e:eta})  is said to have a statistical solution in the space~$C([0, \infty); L^{2}(K))$ 
if there is a probability measure~$\mathfrak{P}^{\e}$ supported on $C([0, \infty); L^{2}(K)\times H^{1}(K))$, and processes $(\hat{u}^{\e}, \hat{\eta}^{\e})\in C([0, \infty); L^{2}(K))$, $\hat{W}$~defined on a new probability space, such that 
\begin{enumerate}
  \item  $\mathcal{D}(\hat{u}^{\e},   \hat{\eta}^{\e})=\mathfrak{P}^{\e}$\,;
  \item $\hat{W}$ is  Wiener processes  distributed the same as~$W$;
 \item $\mathcal{D}(\hat{u}^{\e}(0))=\mathcal{D}(u_{0}^{\e})$\,, $\mathcal{D}(\hat{\eta}^{\e})=\mathcal{D}(\eta^{\e} )$ and $\hat{u}^{\e}(0) $ are independent  from~$\hat{W}$;
 \item the process~$\hat{u}^{\e} $ is a weak solution of~(\ref{e:rBurgers2}) with $\eta^{\e}$~replaced by~$\hat{\eta}^{\e}$.  Here $\hat{\eta}^{\e} $ is a stationary process solving~(\ref{e:eta}) with $W$ replaced by~$\hat{W}$.
  \end{enumerate}

A stationary statistical solution is a statistical solution, a Borel measure $\bar{\mathfrak{P}}^{\e}$\,, which is invariant under the following translation on $C([0, \infty); L^{2}(K)\times H^{1}(K))$
\begin{equation*}
 (u(\cdot), \eta(\cdot))\mapsto  (u(\cdot+\tau),  \eta(\cdot+\tau))\,, \quad \tau\geq 0
\end{equation*}
for $(u, \eta)\in C([0, \infty); L^{2}(K)\times H^{1}(K)) $\,. For a statistical solution of the system of the random equation~(\ref{e:rBurgers2}) coupled with~(\ref{e:eta}), we denote by
$\mathfrak{P}^{\e}_{\tau}=\mathcal{D}(\hat{u}^{\e}(\cdot+\tau), \hat{\eta}^{\e}(\cdot+\tau))$\,, which is also a statistical solution of
the random  equation~(\ref{e:rBurgers2}) coupled with~(\ref{e:eta}). For a stationary
statistical solution~$\bar{\mathfrak{P}}^{\e}$,
\begin{equation*}
\bar{\mathfrak{P}}^{\e}_{\tau}=\bar{\mathfrak{P}}^{\e}\,, \quad \tau\geq 0.
\end{equation*}

The following result establishes a relation between the stationary measure and stationary statistical solution.
\begin{lemma}\label{lem:A1}
If the random   equation~(\ref{e:rBurgers2}) coupled with~(\ref{e:eta}) has a  stationary measure supported on $L^{2}(K)\times H^{1}(K)$, then there is a
 stationary  statistical solution in $C([0, \infty);
L^{2}(K)\times H^{1}(K))$\,.
\end{lemma}
\begin{proof}
The proof is direct by the following observation~\cite{CS08}: Let $\bar{\mathfrak{P}}^{\e}=\mathcal{D}(\bar{\hat{u}}^{\e},  \bar{\hat{\eta}}^{\e})$ be a stationary statistical solution to the random equation~(\ref{e:rBurgers2}) coupled with~(\ref{e:eta}), then 
\begin{equation*}
\bar{\mathfrak{y}}^{\e}=\mathcal{D}(\bar{\hat{u}}^{\e}(0),   \bar{\hat{\eta}}^{\e}(0))
\end{equation*}
  is a stationary measure for the Markov process  defined by the random equation~(\ref{e:rBurgers2}) coupled with~(\ref{e:eta}); Conversely, assume $\bar{\mathfrak{y}}^{\e}$ is a stationary measure of the random equation~(\ref{e:rBurgers2}) coupled with~(\ref{e:eta}),  let $(\bar{u}^{\e},  \bar{\eta}^{\e})$ be a solution of the random equation~(\ref{e:rBurgers2}) coupled with~(\ref{e:eta})  with $\mathcal{D}(\bar{u}^{\e}(0),  \bar{\eta}^{\e}(0))=\bar{\mathfrak{y}}^{\e}$\,, then $\bar{\mathfrak{P}}^{\e}=\mathcal{D}(\bar{u}^{\e},  \bar{\eta}^{\e})$ is a stationary statistical solution of the random equation~(\ref{e:rBurgers2}) coupled with~(\ref{e:eta}).
\end{proof}

For the stochastic  Burgers type  equation~(\ref{e:tau-xi-Burgers}), a  statistical
solution in the space~$C([0, \infty; L^{2}(K))$ is a
probability measure~$\mathbb{P}$
supported on~$C([0, \infty; L^{2}(K))$  and there
are  processes $(\hat{u})\in C([0, \infty; L^{2}(K))$ and~$\hat{W} $ defined on a
new probability space such that
\begin{enumerate}
  \item[i]  $\mathcal{D}(\hat{u})=\mathbb{P}$\,;
  \item[ii] $\hat{W}$ is  Wiener processes  distribute same as~$W$;
 \item[iii] $\mathcal{D}(\hat{u}(0))=\mathcal{D}(u_{0})$ and $\hat{u}(0) $ are independent  from~$\hat{W}$;
 \item[iv] the process $\hat{u}$ is a weak solution of stochastic Burgers type  equation~(\ref{e:tau-xi-Burgers}) with $W$~replaced by~$\hat{W}$.
  \end{enumerate}
Notice that the above definition of statistical solution is a solution to a martingale problem~\cite[Chapter V]{MM88}.


\paragraph{Acknowledgements}
This research was supported by the
Australian Research Council grants DP0988738 and NSFC No. 11371190. Part of the work was done while Wei Wang was working at University of Adelaide.



\end{document}